\documentclass{amsart}
\usepackage{pictex, color}
\newif\ifpdf
	\ifx\pdfoutput\undefined
	\pdffalse 
	\else
	\pdfoutput=1 
	\pdftrue
	\fi

	\ifpdf
	\usepackage[pdftex]{graphicx}
	\else
	\usepackage{graphicx}
	\fi

\theoremstyle{plain}
\newtheorem{Thm}{Theorem}
\newtheorem{Lem}{Lemma}
\newtheorem{Prop}{Proposition}
\newtheorem{Cor}{Corollary}
\theoremstyle{definition}
\newtheorem{Def}{Definition}

\newcommand{\ve}{\varepsilon}
\newcommand{\confrac}[2]{
\frac{\displaystyle{
\strut\hfill{#1}\hfill\;\vrule}}
{\displaystyle{
 \strut\vrule\;\hfill{#2}\hfill}}}
\newcommand{\R}{\mathbb R} \newcommand{\Z}{\mathbb Z}
\newcommand{\N}{\mathbb{N}} 

\title{Metric and arithmetic properties of mediant-Rosen maps}

\author{Cor Kraaikamp}
\address{Technische Universiteit Delft and Thomas Stieltjes Institute of
Mathematics\\ EWI\\ Mekelweg 4\\ 2628 CD Delft, the Netherlands}
\email{c.kraaikamp@tudelft.nl}
\author{Hitoshi Nakada}
\address{Department of Mathematics, Keio University, Yokohama, Japan}
\email{nakada@math.keio.ac.jp}
\author{Thomas A. Schmidt}
\address{Oregon State University\\Corvallis, OR 97331, USA}
\email{toms@math.orst.edu}
\thanks{The second author was supported by Bezoekersbeurs \textbf{B 61-620} of
the Nederlandse Organisatie voor Wetenschappelijk Onderzoek (NWO).   The first and third author thank the Center of Excellence of Keio University.} 
\date{December 1, 2008}

\begin{document}

\begin{abstract}
We define maps which induce mediant convergents of Rosen continued
fractions and discuss  arithmetic and metric properties of mediant
convergents. In particular, we show equality of the ergodic
theoretic Lenstra constant with the arithmetic Legendre constant for
each of these maps.   This value is sufficiently small that the mediant Rosen convergents directly determine the Hurwitz constant of Diophantine approximation of the underlying Fuchsian group.   We thus succeed in giving a continued fractions based verification of these Hurwitz values. 
\end{abstract}

\maketitle

\section{Introduction}
Ergodic properties of a number theoretic transformation can in
certain circumstances be studied by way of transformations which
induce it.      In the
classical setting of the \emph{simple continued fraction} (SCF)
expansion  S.~ Ito ~\cite{I} studied maps corresponding to the 
mediant convergents for exactly this purpose.    Motivated by this classical setting, 
 we call any such
inducing transformation a \emph{mediant} map.  
In this paper we give mediant maps for the Rosen continued
fraction maps, allowing us to extend the work of \cite{Na3} and \cite{BKS}. We discuss some
arithmetic and metric properties of mediant convergents arising
from these maps, in particular using techniques of \cite{Na4} to
show that the Legendre constant---determining membership in the
sequence of approximations of a real number---is equal to the
ergodic theoretic Lenstra constant.

One motivation for this work comes from Diophantine approximation in terms of the Rosen fractions.    Diophantine approximation by simple continued fractions has of course a rich history.      In particular, for the regular continued fraction expansion we have the
following classical {\em Borel result}  (cf.~\cite{DK}): if $x$ has
SCF-expansion $x=[0;a_1,a_2,\dots]$, and convergents $P_n/Q_n$,
$n\geq 0$, and if the approximation coefficients
$\vartheta_n=\vartheta_n(x)$, $n\geq 0$, are defined by
$$
\vartheta_n=\vartheta_n(x)=Q_n^2\left| \, x-\frac{P_n}{Q_n}\right|
,\qquad n\geq 0,
$$
then for every $n\geq 1$ and every irrational $x$ we have
$$
\min (\vartheta_{n-1},\vartheta_n,\vartheta_{n+1}) < 1/\sqrt{5},
$$
and the constant $1/\sqrt{5}$ is sharp. Borel's result,
together with the yet  older  {\em Legendre result} --- 
if $p, q\in \Z$, $q>0$, $\gcd (p,q)=1$,  and 
$\left| x-\frac{p}{q}\right| <1/{2q^2}$  then $p/q$ is a SCF-convergent to $x$
 --- 
implies the classical {\em Hurwitz result}:
 for every
irrational $x$ there are infinitely many rationals $p/q$, such
that
$$
\left| x-\frac{p}{q}\right| < \frac{1}{\sqrt{5}q^2}\,.
$$

A geometric aspect of such Diophantine approximation is expressible in terms of the M\"obius action of the modular group $\text{PSL}(2, \mathbb Z)$.     In the middle of the last century,   this was generalized to approximation by the orbit of infinity under a reasonably large class of Fuchsian groups, see the discussion  on  pp. ~334--336 of \cite{LeBk}.    Some thirty years after Rosen \cite{R} introduced his continued fractions to study elements of the Hecke triangle groups (a family of Fuchsian groups including the modular group),     Lehner  \cite{Le1, Le2}  used these continued fractions to begin the study of the quality of approximation by the orbit of infinity  under each of the Hecke groups.   His goal was to use the Rosen fractions to determine   the  analog of the Hurwitz constant for this approximation.

These Hurwitz constants were finally determined by Haas and Series \cite{Ha-Ser}, using techniques of hyperbolic geometry.    Let $G_k$ denote the {\em Hecke group} of index $k\ge 3$, generated by $z \mapsto -1/z$ and $z \mapsto z + \lambda_k$ with $\lambda_k = 2 \cos \pi/k$.   The {\em $G_k$-rationals},  denoted $G_{k}(\infty)$,  is the orbit of infinity under this group; a real number is called a  {\em $G_k$-irrational} if it is not in this orbit.

 \begin{Thm}{\rm (Haas and Series)}\label{Haas+Series}
 For $x\in \R \setminus G_{k}(\infty)$, let
$$
\mu_k(x)=\inf \{ h\, ;\, \left| \, x - \frac{a}{c}\right| <
\frac{h}{c^2}\,\, \text{has infinitely many solutions
$\frac{a}{c}\in G_{k}(\infty)$} \} ,
$$
and set $ C(k)=\sup \{ \mu_k(x)\, ;\, x\in \R \setminus
 G_{k}(\infty)\}\,$. Then
$$
C(k)=\begin{cases} 1/2 & \text{if $k$ is even;}\\
\dfrac{\displaystyle 1}{\displaystyle 2\sqrt{\left(
1-\lambda_k/2\right)^2+1}} & \text{if $k$ is odd.}
\end{cases}
$$
\end{Thm}

A main goal of this paper is to give a new proof of this result using continued fraction methods and thus to complete Lehner's program.   (Since $k=3$ is the index of the classical case, we will restrict our considerations to $k\ge 4$ throughout.)  One could imagine that a continued fractions proof proceeds by way of analogs of Borel and Legendre results.   Indeed, ~\cite{KSS} gives a Borel-type result: 

 \begin{Thm}{\cite{KSS}}\label{KSSBorel}
 For $x\in \R \setminus G_{k}(\infty)$ and
$(p_n/q_n)_{n\geq 1}$ the corresponding Rosen convergents of
$x$, one has that
$$
\theta_n(x) := q_n^2 \left| \, x - \frac{p_n}{q_n}\, \right| <
C(k)\quad \text{infinitely often,}
$$
and the constant $C(k)$ is optimal.
\end{Thm}
 \noindent
Furthermore, due to results of Nakada \cite{Na3}, a Legendre-result
exists for the Rosen fractions.   Unfortunately,  this analog of the Legendre constant is strictly {\em less}   than $C(k)$, so Theorem \ref{KSSBorel} does not immediately imply Theorem \ref{Haas+Series}.  That is, these two results do not rule out the existence of some constant $D$, with $D < C(k)$, such that for some $G_k$-irrational $x$ there exist
infinitely many $G_k$-rationals $a/c$, which are \emph{not} Rosen
convergents of $x$, but do satisfy
$$
\left| x - \frac{a}{c}\right| < \frac{D}{c^2}.
$$

To address this difficulty,  we introduce the mediant Rosen maps (see Section \ref{sec:definitions_and_fundrel}).         For each $k$, there is an 
 $\ell_{k} > 0$ such that
(i)  for any $G_{k}$-irrational $x$ and any finite $G_{k}$-rational $a/c$, 
\[
\left| x \, - \, \frac{a}{c} \right| \, < \, \frac{\ell_{k}}{c^2}
\]
implies $a/c$ is either a Rosen convergent $p_n/q_n$ for some $n \ge
0$, or a mediant Rosen convergent $u_{n,l}/v_{n, l}$ of $x$;  and,
(ii) for any $C > \ell_{k} $, there exist $x$ and $a/c$ such that
\[
\left| x \, - \, \frac{a}{c} \right| \, < \, \frac{C}{c^2}
\]
and $a/c$ is neither a Rosen convergent nor a mediant Rosen
convergent.  We call $\ell_{k} > 0$ the {\em Legendre constant} for
mediant Rosen convergents (of index $k$).

We turn to ergodic theory to determine the value of this Legendre constant of Diophantine approximation. 
In the setting of the simple continued fraction map,  H.W. ~Lenstra, ~Jr.,  conjectured the value  of the endpoint of linearity in the average value of small
approximation coefficients (for almost every $x$);  this conjecture was confirmed by Bosma {\em et al} \cite{BJW}.  
 In Section ~\ref{Legendre+Lenstra} we define an analogous value,    ${\mathcal L}_{k}$,  the {\em Lenstra constant}
for the mediant Rosen map (of index $k$).       (Haas \cite{H} has recently shown that Lenstra constants of continued fraction type maps  are related to universal behavior of geodesic
excursions into cusps of hyperbolic surfaces.)      Nakada \cite{Na4} has proved that whenever a continued fraction map
has a Legendre constant, then it also has a Lenstra constant, and if the map is ergodic with respect to a  {\em finite} invariant measure, then these two are equal.    (In the SCF case, the Legendre and Lenstra constants both equal $1/2\,$.)     We show that each Rosen mediant map is ergodic with respect to an {\em infinite} invariant measure and hence cannot directly invoke the result of \cite{Na4}.   Instead,  we use explicit planar extensions and the ratio ergodic theorem (see for example \cite{A})  to  prove the equality of the two constants and to determine their common value by actually evaluating the ergodic theoretic Lenstra constant.

 \begin{Thm}\label{legLenEqual}  Fix $k \ge 4$, and let $\ell_k$ and $\mathcal L_k$ be the Legendre and Lenstra constants for the Rosen mediant maps.  Furthermore,  when $k$ is even then  let $R$ equal 1, otherwise let $R$ be the positive root of $x^2 +(2- \lambda_k)x - 1$.   Then  for $k>4$
 \[ \ell_k = \mathcal L_k = \lambda-R\;.\]  
 Also,  $\ell_4 = \mathcal L_4 = \sqrt{2}/2$.    In particular,  in all cases we have 
 \[ \ell_k> C(k)\;.\]   
\end{Thm}

Therefore, any $G_k$-rational $a/c$ satisfying
$$
\left| \, x - \frac{a}{c}\right| < \frac{C(k)}{c^2},
$$
which is not a Rosen convergent \emph{must} be a mediant
convergent of $x$.    Our next result gives a {\em witness} to the optimality of $C(k)$;  we use $\Theta_n(x)$ to denote the coefficients of approximation for the Rosen median maps,  in direct analogy to the $\theta_n$ above (see Equation ~\eqref{approxCoeffs} below for a definition).

 \begin{Thm}\label{witness}  Fix $k \ge 4$.   Let  $C(k)$  and $R$ be as above.   Then $\tau_0 = R-\lambda$ is such that for any $C < C(k)$ 
\[
\Theta_n (\tau_0)<C,\quad \text{for at most finitely many}\; n\geq
0\,.
\]
 \end{Thm}
We note that Theorems \ref{KSSBorel}, \ref{legLenEqual}  and \ref{witness}  give a continued fraction proof of Theorem \ref{Haas+Series}.

\subsection{Outline} In the next section, we introduce the mediant algorithm.  In Section ~\ref{sec:definitions_and_fundrel}, we give the underlying mediant maps.  Section ~\ref{sec:NatExt} is devoted to the construction of planar natural extensions for these, and to the study of their basic ergodic properties.    Section ~\ref{sec:hurwitz} provides the proof of Theorem \ref{witness}.  Definitions of the Legendre and Lenstra constants for the mediant maps appear in Section ~\ref{Legendre+Lenstra}, where their equality  is proven.    In Section ~\ref{the-Lenstra-constant}, we  evaluate the Lenstra constants.

\subsection{Thanks}  We thank the referee for a careful reading and sound advice for improving the presentation of these results:    in particular for insisting on clarifying material which lead us to include Subsection ~\ref{fritzToRescue}.

\section{Mediants of Rosen Fractions}\label{sec:Med}

\begin{figure}
\begin{center}
\scalebox{0.8}{\includegraphics{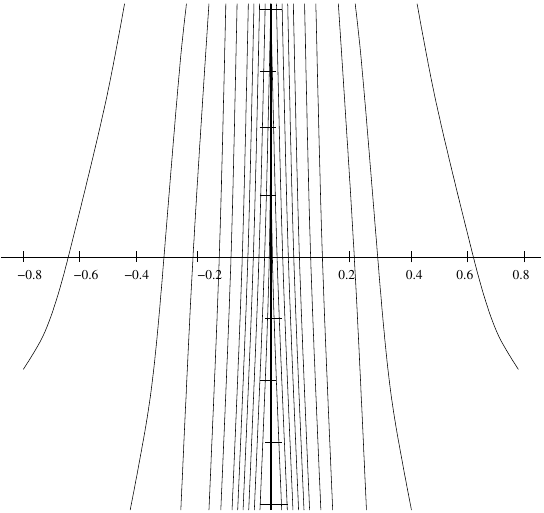}}
\caption{Graph of $T_k(x)$, k=5}
\label{thatsTfig}
\end{center}
\end{figure}

Throughout this paper, $\lambda_k = 2 \cos \frac{\pi}{k}$ and
${\mathbb I}_k = [-\lambda/2, \lambda/2\,)$ for $k
\ge 3$. For a fixed integer $k \ge 3$, the Rosen continued fraction
map is defined by
\[
T_{k}(x) = \begin{cases} \left| \frac{1}{x} \right| \, - \,
\lambda \lfloor \left| \,\frac{1}{\lambda x} \right| + \frac{1}{2}
\rfloor & x \ne 0; \\
  0   &  x = 0
\end{cases}
\]
for $x \in {\mathbb I}_k$; here and below, we  omit the index
``$k$" whenever it is clear from context.  We define
\[
\ve_n(x) = \mbox{sgn}(T^{n-1}x) \qquad \mbox{and} \qquad
r_n(x) = r(T^{n-1}x)
\]
with
\[
\ve (x) = \mbox{sgn}(x) \qquad \mbox{and} \qquad
r(x) = \left\lfloor \, \left| \frac{1}{\lambda x} \right| + \frac{1}{2}
\right\rfloor .
\]
Then, as Rosen showed in \cite{R},  we have the Rosen continued fraction expansion of $x$ as follows
\[
x \, = \, \confrac{\ve_1(x)}{ \lambda r_1(x)} \, + \,
          \confrac{\ve_2(x)}{ \lambda r_2(x)} \, + \,
          \cdots +\confrac{\ve_n(x)}{ \lambda r_n(x)} \, + \, \cdots \,
          ,
\]
which is denoted by $x=[\, \ve_1(x):r_1(x),\,
\ve_2(x):r_2(x),\ldots ,\, \ve_n(x):r_n(x),\ldots ]\,$. Here the
expansion terminates at a finite term if and only if $x$ is a
parabolic point of  $G_k$, thus if and only if $x$ is a $G_k$-rational
(see~\cite{R}).  As usual we can define the convergents $p_n/q_n$
of $x\in {\mathbb I}_k$ by
\[
\begin{pmatrix}
p_{-1}  &  p_0  \\
q_{-1}  &  q_1
\end{pmatrix}
\, = \,
\begin{pmatrix}
1  &  0  \\
0  &  1
\end{pmatrix}
\]
and
\[
\begin{pmatrix}
p_{n-1}  &  p_{n}  \\
q_{n-1}  &  q_{n}
\end{pmatrix}
\, = \,
\begin{pmatrix}
0 &  \ve_{1}  \\
1  &  \lambda r_{1}
\end{pmatrix}
\begin{pmatrix}
0 &  \ve_{2}  \\
1  &  \lambda r_{2}
\end{pmatrix}
\cdots
\begin{pmatrix}
0 &  \ve_{n}  \\
1  &  \lambda r_{n}
\end{pmatrix}
\]
for $n \ge 1$. From this definition it is easy to see that $\left|
p_{n-1}q_{n} \, - \, q_{n-1}p_n \right| =1$, and that we have the
well-known recurrence relations
\begin{eqnarray*}
p_{-1}=1; p_0=0; & p_n=\lambda r_np_{n-1}+\varepsilon_np_{n-2},\,
n\geq 1\\
q_{-1}=0; q_0=1; & q_n=\lambda r_nq_{n-1}+\varepsilon_nq_{n-2},\,
n\geq 1.
\end{eqnarray*}
It also follows that
\[
\begin{pmatrix}
p_{n-1}  &  q_{n-1}  \\
p_{n}  &  q_{n}
\end{pmatrix}
\, = \,
\begin{pmatrix}
0 &  1  \\
\ve_{n}  &  \lambda r_{n}
\end{pmatrix}
\begin{pmatrix}
0 &  1  \\
\ve_{n-1}  &  \lambda r_{n-1}
\end{pmatrix}
\cdots
\begin{pmatrix}
0 &  1  \\
\ve_{1}  &  \lambda r_{1}
\end{pmatrix} ,
\]
giving
\begin{equation}
\frac{p_n}{q_n} \, = \,
\confrac{\ve_1}{ \lambda r_1} \, + \,
          \confrac{\ve_2}{ \lambda r_2} \, + \,
          \cdots +\confrac{\ve_n}{ \lambda r_n}
\end{equation}
and
\begin{equation}
\frac{q_{n-1}}{q_n} \, = \,
\confrac{1}{ \lambda r_n} \, + \,
          \confrac{\ve_{n}}{ \lambda r_{n-1}} \, + \,
          \cdots + \confrac{\ve_{2}}{ \lambda r_{1}} .
\end{equation}
Since Hecke groups are discontinuous groups, the (parabolic) value
$ p_n/q_n$  uniquely determines $q_{n}$ up to sign; we can and do
assume $q_{n}$ to be positive.

From the definition of $p_n/q_n$, we have
\[
\frac{p_{n+1}}{q_{n+1}} \, = \, \frac{r_{n+1} \lambda
p_n+\varepsilon_{n+1}p_{n-1}} {r_{n+1}\lambda
q_n+\varepsilon_{n+1}q_{n-1}} .
\]
If $r_{n+1} > 1$, then---following~\cite{I} in the SCF-case---we
can interpolate between $p_n/q_n$ and $p_{n+1}/q_{n+1}$ by
\begin{equation}\label{eq:vees}
\frac{u_{n, l}}{v_{n,l}}=\frac{l \lambda
p_n+\varepsilon_{n+1}p_{n-1}}{l \lambda q_n+
\varepsilon_{n+1}q_{n-1}} , \quad 1 \le l < r_{n+1},
\end{equation}
and call $u_{n, l}/v_{n, l}$ the $l$-th mediant convergent of $x$
(of level $n$); note that such a mediant does not exist in case
$r_{n+1}=1$. In the next section, we define a map which induces
mediant convergents of Rosen continued fractions.

\section{Mediant Maps and Convergents}\label{sec:definitions_and_fundrel}
\subsection{Full Mediant Map}
For each fixed $k$,  put ${\mathbb J} = {\mathbb J}_k = [ - \lambda/2, \,
\frac{2}{\lambda} )$.

For ease of notation,  we  let $\begin{pmatrix} a & b \\ c & d
\end{pmatrix} (x)$ denote  $\displaystyle{\frac{ax + b}{cx + d}}$.    This allows us to show a ``factorization'' of the map $T_k$, leading to our definition of the Rosen mediant maps.  

  The following is trivially verified.
\begin{Lem} The Rosen map $T_k$ can be expressed in the following
manner.
\[
T_k(x)=\left\{ \begin{array}{cl}
         \begin{pmatrix}
         \lambda& 1  \\
         -1     & 0
         \end{pmatrix}
   (x), &  x\in [-\frac{\lambda}{2}, -\frac{2}{3 \lambda}); \\
& \\
         \begin{pmatrix}
         t \lambda& 1  \\
         -1     & 0
         \end{pmatrix}
   (x), &  x\in [-\frac{2}{(2t-1)\lambda},
                     - \frac{2}{(2t+1)\lambda}),\, t\in \N_{\geq 2}; \\
& \\
         \begin{pmatrix}
         -t \lambda& 1  \\
         1     & 0
         \end{pmatrix}
   (x), &  x \in (\frac{2}{(2t+1)\lambda},
                      \frac{2}{(2t-1)\lambda}],\, t\in \N_{\geq 2};\\
& \\
         \begin{pmatrix}
         -\lambda& 1  \\
         1     & 0
         \end{pmatrix}
   (x), &  x\in (\frac{2}{3 \lambda},
   \frac{\lambda}{2}).
\end{array}\right.
\]
\end{Lem}

\bigskip
\begin{Def}\label{def:mats}  We define the following
matrices.
\[
U_{-} \, = \,  \begin{pmatrix}
         0 & -1 \\
         1 & \lambda         \end{pmatrix}
         \;; \quad U_{+} \, = \,
         \begin{pmatrix}
         0 & 1 \\
         1 & \lambda         \end{pmatrix}
         \;; \quad V_{-} \, = \,
         \begin{pmatrix}
         -1 & 0 \\
         \lambda& 1
         \end{pmatrix}
          \;; \quad V_{+} \, = \,
         \begin{pmatrix}
         1 & 0 \\
         \lambda& 1
         \end{pmatrix}
          \;.
\]
\end{Def}

The inverses of these are of course:
\[
U_{-}^{-1} \, = \,
         \begin{pmatrix}
         \lambda& 1 \\
         -1 & 0
         \end{pmatrix}
         \;; \quad U_{+}^{-1} \, = \,
         \begin{pmatrix}
         -\lambda& 1 \\
         1 & 0
         \end{pmatrix}
         \;; \quad V_{-}^{-1} \, = \,
         \begin{pmatrix}
         -1 & 0 \\
         \lambda& 1
         \end{pmatrix} ,
\]
and
\[
         V_{+}^{-1} \, = \,
         \begin{pmatrix}
         1 & 0 \\
         -\lambda& 1
         \end{pmatrix}
         \;.
\]

\medskip\
The next lemma is verified by direct computation.

\begin{Lem}\label{trivMatCalc}   The following
equalities hold:
\[
         \begin{pmatrix}
         t \lambda& 1 \\
         -1 & 0
         \end{pmatrix}
  \, = \, U_{+}^{-1} \cdot V_{+}^{-(t-2)} \cdot V_{-}^{-1} \; \;
  \text{and} \;\;
         \begin{pmatrix}
         -t \lambda& 1 \\
         1 & 0
         \end{pmatrix}
  \, = \, U_{+}^{-1} \cdot V_{+}^{-(t-1)}\;.
\]
\end{Lem}

\medskip

\begin{Def}\label{def:RosMed}   The {\em Rosen mediant map} is 
 
\[
S_k(x)=\begin{cases}
U_{-}^{-1}(x), & x \in [-\frac{\lambda}{2}, -\frac{2}{3\lambda})\;; \\
\\
V_{-}^{-1}(x), & x \in [-\frac{2}{3\lambda}, 0)\;; \\
\\
V_{+}^{-1}(x), & x \in (0, \frac{2}{3\lambda}]\;; \\
\\
U_{+}^{-1}(x), & x \in (\frac{2}{3\lambda}, \frac{2}{\lambda}) \;.
\end{cases}
\]
\end{Def}

In the case of $k=3$, $T_3$ is the classical nearest integer
continued fraction map, and $S_3$ is the mediant map as defined by
R.~Natsui in~\cite{Nt}. In the sequel we always assume that $k \ge
4$.    Here also, we suppress the index $k$ when discussing these maps.

\begin{figure}
\begin{center}
\includegraphics{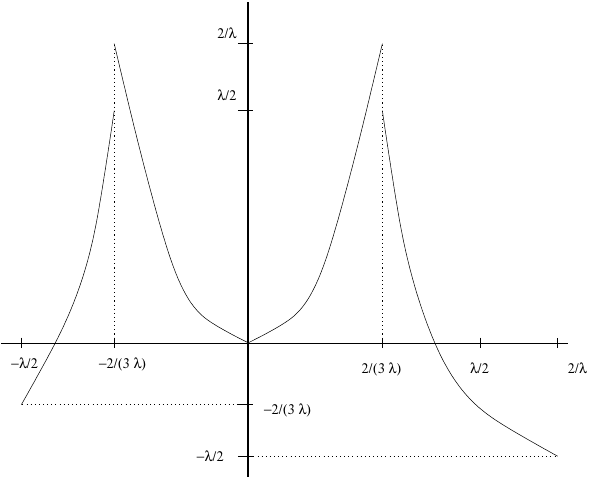}
\caption{Graph of $S_k(x)$, k=5}
\label{interMedMapFig}
\end{center}
\end{figure}
\medskip\
Direct calculation also shows the following.
\begin{Lem} Viewed as a linear fractional transformation, the matrix
$V_{-}^{-1}$ is a monotonic  decreasing  bijective map from $\left( -\frac{2}{(2l -1)
\lambda}, -\frac{2}{(2l +1) \lambda}\right)$ to $\left( \frac{2}{(2l
-1) \lambda}, \frac{2}{(2l -3) \lambda}\right)$.
Similarly, the linear fractional transformation $V_{+}^{-1}$   is a monotonic increasing bijective map from  $\left( \frac{2}{(2l +1) \lambda},
\frac{2}{(2l -1) \lambda}\right)$ to $\left(
\frac{2}{(2l -1) \lambda}, \frac{2}{(2l -3) \lambda}\right)$.
\end{Lem}

\bigskip
Finally, we have that $T(x)$ is induced by $S(x)$.  
\begin{Lem}\label{trueMediant}   For each $x \in {\mathbb I}_k$,
let $\ell (x)$ be defined as follows.
\[
\ell (x) := \min \left\{ \ell \ge 0 :\, S^{\ell}(x) \in
[-\frac{\lambda}{2}, -\frac{2}{3\lambda})
             \cup (\frac{2}{3\lambda}, \frac{2}{\lambda}) \right\}\;.
\]
Then for each $x\in {\mathbb I}_k$, one has the following
equality:
\[
S^{\ell (x) + 1} (x)=T(x).
\]
\end{Lem}

\begin{proof} This follows from our definitions and of Lemmas of this section. \end{proof}
\bigskip
\subsection{Mediant Convergents}
For $x\in {\mathbb I}_k$ and $i \in \mathbb N$,  we let
\[
M_{i}=\left\{ \begin{array}{ccl}
U_{-} & \mbox{if} & S^{i-1}(x) \in [-\frac{\lambda}{2}, -\frac{2}{3\lambda}); \\
\\
U_{+} & \mbox{if} & S^{i-1}(x) \in (\frac{2}{3\lambda}, \frac{2}{\lambda});\\
\\
V_{-} & \mbox{if} & S^{i-1}(x) \in [-\frac{2}{3\lambda}, 0);\\
\\
V_{+} & \mbox{if} & S^{i-1}(x) \in (0, \frac{2}{3\lambda}];\\
\\
\text{Id}   & \mbox{if} & S^{i-1}(x) \, = \, 0.
\end{array} \right.
\]
for $i \ge 1$, where as usual $\text{Id}$ denotes the identity.

\bigskip

Then we have a sequence of matrices from $x$ which is denoted by
\begin{equation}\label{mSeq}
x \, \sim \, M_1 M_2 \cdots M_j \cdots
\end{equation}
Of course, the mediant map acts as a shift on each such sequence.

\begin{Def}  Fix $x\in {\mathbb I}_k$ and consider the above
sequence of $M_i$.     Let $k_1, k_2, \dots$ be the increasing sequence of
indices for which $M_{k_i}\in \{\,U_- \, , \, U_+\,\}$.
\end{Def}

The following lemma records the fact that the sequence of Rosen convergents
$p_n/q_n$ of $x\in {\mathbb I}_k$ is a subsequence of the sequence
$u_{n, l}/v_{n,vl}$, $n\geq 1$, of mediant convergents of $x$.
\begin{Lem} For each $x \in {\mathbb I}_k$, consider the
corresponding sequence of Equation~{\rm (\ref{mSeq})}. Then, for
each $k_m$ as above, one has the following equality.
\[
  M_1 M_2 \cdots M_{k_m} \, = \,
         \begin{pmatrix}
         p_{m-1} & p_m \\
         q_{m-1} & q_m
         \end{pmatrix}
        \;.
\]
Furthermore, $k_{m+1} = k_m + r_{m+1}$ where $(\ve_{m}: r_{m} )$ is
the m-th coefficient of the Rosen continued fraction expansion of
$x$.
\end{Lem}

 \begin{Def} Let $x$, $m$ and $k_m$ be as above.
For each integer $l$ with $0<l<r_{m+1}$, we define
\[
\left\{ \begin{array}{c}
u_{m,l} \, = \, l \cdot \lambda\cdot p_m \, + \, \ve_{m+1}p_{m-1}\;; \\
\\
v_{m,l} \, = \, l \cdot \lambda\cdot q_m \, + \,
\ve_{m+1}q_{m-1}\;,
\end{array} \right.
\]
and call $\dfrac{u_{m,l}}{v_{m,l}}$ an {\em $l$-th mediant
convergent} of $x$.
\end{Def}

We have the following result.
\begin{Prop} With notation as above, we have
\[
M_1 \cdots M_{k_m} \cdots M_{k_m+l} \, = \,
         \begin{pmatrix}
         u_{m,l} & p_m \\
         v_{m,l} & q_m
         \end{pmatrix}
\]
for $1 \le l < r_{m+1}$.
\end{Prop}

By this proposition, we see that the sequence $(M_1 \cdots
M_i(\infty) \, : \, i \ge 1)$ is
\[
\begin{aligned}
\frac{u_{0,1}}{v_{0,1}} , \, \frac{u_{0,2}}{v_{0,2}} , \, \cdots \,
&, \, \frac{u_{0,r_1-1}}{v_{0,r_1-1}} \, , \, \frac{p_0}{q_0} \, ,
\, \frac{u_{1,1}}{v_{1,1}} , \, \frac{u_{1,2}}{v_{1,2}} , \, \cdots
\, , \, \frac{u_{1,r_2-1}}{v_{1,r_2-1}} \, , \,
\frac{p_1}{q_1} \, , \, \cdots \, \\
\\
&\cdots \,,\, \frac{u_{n,1}}{v_{n,1}} , \, \frac{u_{n,2}}{v_{n,2}} ,
\, \cdots \, , \, \frac{u_{n,r_{n+1}-1}}{v_{1,r_{n+1}-1}} \, , \,
\frac{p_n}{q_n}, \, \cdots
\end{aligned}
\]
It is easy to see that
\[
    x \, = \,
    \left(
        \begin{array}{cc}
        u_{m,l} & p_m \\
        v_{m,l} & q_m
        \end{array}
   \right) \left( S^n(x)\right) ,
\]
for $n = k_{m} + l$.  We put $x_n = S^n(x)$. It follows that
\begin{equation}\label{eq:distance-x-to-mediant}
\left| x \, - \, \frac{u_{m,l}}{v_{m,l}}\right| = \left|
\frac{u_{m,l}x_n + p_m}{v_{m,l}x_n + q_m} \, - \,
\frac{u_{m,l}}{v_{m,l}}\right| = \frac{1}{v_{m,l}^2 ( x_n - (-
\frac{q_m}{v_{m,l}})\,)} ,
\end{equation}
where $-\frac{q_m}{v_{m,l}} = (M_1 \cdots M_{n})^{-1}(\infty)$.
We recall that $(M_1 \cdots M_{n})^{-1}$ is the linear fractional
transformation which defines $S^n$.
It also follows that
\begin{equation}\label{distance}
\left| x - \frac{p_{m-1}}{q_{m-1}}\right| =\frac{1}{q_{m-1}^2
(x_{k_m} \, - \, (-\frac{q_m}{q_{m-1}}))}
\end{equation}
where $-q_m/q_{m-1} = (M_1 \cdots M_{k_m})^{-1}(\infty)$.
Consequently, the distribution of
$$
\left( \, (M_1 \cdots
M_{n})^{-1}(x) \, - \, (M_1 \cdots M_{n})^{-1}(\infty) \, : \, n \ge
1 \right)
$$
determines the distribution of the error (after normalization by
the square of the denominator) of the principal and the mediant
convergents.\smallskip\

If $k=3$, since $\lambda_3 = 1$, it is easy to see that
\[
\begin{pmatrix} u_{m,1} \\ v_{m, 1} \end{pmatrix} \, = \,
\begin{pmatrix} u_{m-1,r_m - 1} \\ v_{m-1, r_m -1} \end{pmatrix}
\]
when $\ve_{m+1} = -1$. This equality never holds for $k \ge 4$.
Indeed, since $\lambda_k > 1$ for $k \ge 4$, we have---due
to~(\ref{eq:vees}),
\[
v_{m,1} = \lambda q_m - q_{m-1} > q_m - \lambda q_{m-1} = v_{m-1,
r_m - 1} .
\]
This implies that all values in
\[
\left\{ \frac{u_{m,l}}{v_{m,l}}, \, \frac{p_m}{q_m} \right\}_{m
\ge 0}
\]
are different from each other.
\section{Natural extensions}\label{sec:NatExt}
Since the work of \cite{NIT}, planar natural extensions of continued fractions maps 
have provided a significant tool in the study of ergodic properties of number theoretic transformations.  
 In this section we construct planar natural extensions of the Rosen mediant maps.

In~\cite{Roh}, Rohlin introduced and studied the concept of
\emph{natural extension} of a dynamical system. A
natural extension of $({\mathbb J}, {\mathcal B},
\mu,T)$ is an invertible dynamical system
$(\Omega, {\mathcal B}_{\Omega}, \rho, {\hat
T})$, which contains $({\mathbb J}, {\mathcal B},
\mu,T)$ as a factor, such that the Borel
$\sigma$-algebra ${\mathcal B}_{\Omega}$ of $\Omega$ is
the smallest ${\hat T}$-invariant $\sigma$-algebra
that contains $\pi^{-1}({\mathcal B})$, where $\pi$
is the factor map. A natural extension is unique up to
isomorphism.   With notation defined below, we have the following result.


\begin{Thm}\label{thm:natural-extension}
For $k \ge 4$,  the dynamical system $(\Omega^{*},
\bar{\mathcal B}, \nu,{\hat S})$ is the
natural extension of the dynamical system $({\mathbb J}_k,
{\mathcal B}, \mu_k,S)$.
\end{Thm}

 There are various ways to verify that a planar system is the natural extension of a given interval map; one way is to follow the proof from ~\cite{Na1}, where the second author shows that the two-dimensional 
regions he finds are indeed the natural extensions of his $\alpha$-expansions. His 
proof closely follows ~\cite{Roh}.  Here, however,  we turn to F. ~Schweiger's ~\cite{Schw}  formalization of the ideas of ~\cite{NIT} to verify that we have found the natural extension.

In our construction, we rely on \cite{BKS}.  However, we proceed slightly differently. There the
explicit natural extension map is   
\begin{equation}\label{eq:BKSextension}
{\mathcal T}(x,y)=\left( T(x), \frac{1}{\lambda \lfloor
|\frac{1}{\lambda x}|+\frac{1}{2}\rfloor + {\text{sgn}} (x)
y}\right),
\end{equation}
which is locally of the form given by 
\[ M = \begin{pmatrix}
a  & b \\ c  &  d
\end{pmatrix} \;\; \text {sending}\;\; (x,y)\;\; \text{ to}\;\;
\left(\, \dfrac{a x + b}{c x + d}, \, \dfrac{d y -c}{-b y + a}\,\right)\;;\]
an elementary calculation thus shows that this map has 
invariant measure 
$\frac{{\rm dx}\, {\rm d}y}{(1+xy)^2}$, up to normalizing
constants.  
Here, we use the more natural action directly related to hyperbolic geometry:
\[
(x,y) \mapsto
\left(\, \dfrac{a x + b}{c x + d}, \, \dfrac{a y + b}{c y + d}\,\right)\;.
\]
These maps are conjugate, using $(x,y) \mapsto (x, -1/y)$,  thus there is no loss in proceeding in
our manner.     The invariant measure for our map is  well known to be $\frac{{\rm d}x\, {\rm d}y}{(x - y)^2}\,$.
The domain $\Omega\,$ of $\mathcal T$  is defined (depending on parity) in Theorem~3.1 and Theorem~3.2
of~\cite{BKS}; up to measure zero,  $\mathcal T(\, \Omega\,) = \Omega\,$.   We let 
\begin{equation}\label{original-Rosen-region-transformed}
\Omega_{0} = \{ (x, y) : \, (x, - 1/y) \in \Omega \}
\end{equation}
 and hence $\Omega_0$ is an isomorphic copy of the
region of the natural extension of the Rosen map $T$.

 The parity of the index $k$ is significant,  as in particular displayed by
the orbit of $\pm \frac{\lambda}{2}$ under
$S$, we thus discuss the even index case and the odd index case
separately.

\subsection{Planar system in the even index case: $k = 2\ell$}
We recall
some notations from~\cite{BKS}.  Let
$$
\phi_{0} = - \frac{\lambda}{2},\qquad \phi_{j} = T^{j} \left( -
\frac{\lambda}{2} \right), \,\, 0 \le j \le \ell -1,\qquad
\text{and}\quad \phi_{\ell -1} = 0,
$$
and
$$
L_1 = \frac{1}{\lambda + 1},\qquad L_{j} = \frac{1}{\lambda -
L_{j-1}}, \,\, 2 \le j \le \ell -1,\qquad \text{and}\quad 1 =
\frac{1}{\lambda - L_{\ell -1}}.
$$
It is in terms of  these various $\phi_i$ and $L_i$ that  \cite{BKS}  define the domain $\Omega$.

Now,  for $1 \le j \le \ell -1$, let
\[
\left\{ \begin{array}{ccl}
J_j & = &  [\phi_{j-1}, \, \phi_j) \\
J_{\ell} & = & [0, \frac{\lambda}{2} ) \\
J_{\ell +1} & = & [\frac{\lambda}{2}, \, \frac{2}{\lambda} )
\end{array} \right.
\quad \text{and}\qquad \left\{ \begin{array}{ccl}
\bar{K}_j & = &  [-\infty, \, - \frac{1}{L_j}\,] \\
\bar{K}_{\ell} & = & [-\infty , \, 0\, ]\\
\bar{K}_{\ell +1} & = & [-1, \, 0],
\end{array} \right.
\]
and let $K_{j}' = \bar{K}_j$ for $1\le j \le \ell - 1$ as well as $K_{\ell}' = [-\infty, -1]$.  
Solving,   we find
\[
\Omega_0 = \bigcup_{j=1}^{\ell} \left( J_j \times {K}_{j}'
\right) .
\]
We define the region (of the natural extension for the mediant map)
\[
\Omega^{\ast} = \bigcup_{j=1}^{\ell+1} \left( J_j \times \bar{K}_j
\right) ,
\]
see Figure~\ref{figOm8}. The map $\hat{S} : \Omega^{\ast} \to
\Omega^{\ast}$ is given by
\[
\hat{S}(x, y) = \left( \, M_1^{-1}(x), \, M_1^{-1}(y) \,\right)\,,
\]
where  $M_1=M_1(x)$ as in Equation \eqref{mSeq}.  In particular,   the projection onto the first coordinate is indeed
$M_1^{-1}(x) = S(x)\,$.

\noindent
\begin{figure}[h]
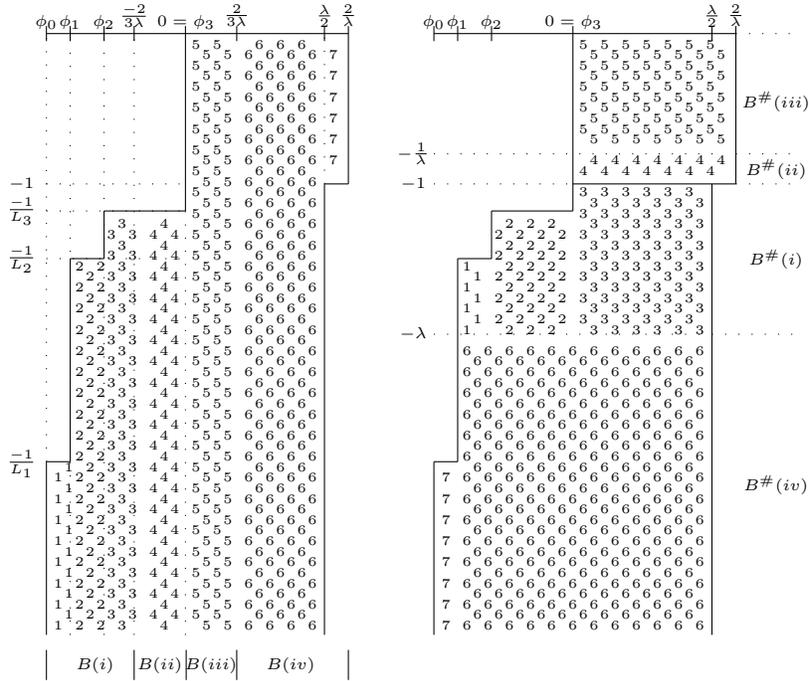

\noindent
\begin{tabular}{cc}
$$
\beginpicture
    \setcoordinatesystem units <0.2cm,0.2cm>
    \setplotarea x from -12 to 12, y from -35 to 1
    \putrule from -9.24 -41 to -9.24 -43
     \put {\tiny{$B(i)$}} at -6  -42
    \putrule from -3.41 -43 to  -3.41 -41
    \put {\tiny{$B(ii)$}} at -1.7  -42
    \putrule from 0.05 -43 to  0.05 -41
    \put {\tiny{$B(iii)$}} at 1.7  -42
    \putrule from 3.41 -43 to  3.41 -41
    \put {\tiny{$B(iv)$}} at 7  -42
    \putrule from 10.83 -43 to  10.83 -41
    \put {\tiny{$0=\phi_3$}} at 0.05 0.8
    \put {\tiny{$\phi_0$}} at -9.24 0.8
    \put {\tiny{$\phi_1$}} at -7.66 0.8
    \put {\tiny{$\phi_2$}} at -5.41 0.8
    \put {\tiny{$\frac{-2}{3 \lambda}$}} at -3.41 1.2
    \put {\tiny{$\frac{2}{3 \lambda}$}} at 3.41 1.2
    \put {\tiny{$\frac{\lambda}{2}$}} at 9.24 1.2
    \put {\tiny{$\frac{2}{\lambda}$}} at 10.83 1.2
    \put {\tiny{$\frac{-1}{L_1}$}} at -10.9  -28.6
    \put {\tiny{$\frac{-1}{L_2}$}} at -10.9  -15
    \put {\tiny{$\frac{-1}{L_3}$}} at -10.9  -11.8
    \put {\tiny{$-1$}} at -10.9  -10
    \putrule from -9.24 -0.5 to -9.24 0.5
    \putrule from -7.66 -0.5 to -7.66 0.5
    \putrule from -5.41 -0.5 to -5.41 0.5
    \putrule from -3.41 -0.5 to -3.41 0.5
    \putrule from  3.41 -0.5 to  3.41 0.5
    \putrule from 9.24 -0.5 to 9.24 0.5
    \putrule from -9.24 0 to 10.83 0
    \putrule from -9.24 -40 to -9.24 -28.48
    \putrule from -9.24 -28.48 to -7.66 -28.48
    \putrule from -7.66 -28.48 to -7.66 -14.97
    \putrule from -7.66 -14.97 to -5.41 -14.97
    \putrule from -5.41 -14.97 to -5.41 -11.8
    \putrule from -5.41 -11.8 to 0 -11.8
    \putrule from 0 -11.8 to 0 0.5
    \putrule from  9.24 -10 to 10.83 -10
    \putrule from 10.83 -10 to 10.83 0.5
    \putrule from 9.24 -40 to 9.24 -10
    \setshadesymbol <.1pt,.1pt,.1pt,.1pt> ({\fiverm 1}) \setshadegrid span <4pt> 
    \vshade  
    -9 -40 -28.6   -8.4 -40 -28.6   -7.6 -40 -28.6 
     /
   \setshadesymbol <.1pt,.1pt,.1pt,.1pt> ({\fiverm 2}) \setshadegrid span <4pt> 
    \vshade 
   -7.6 -40 -15    -6 -40 -15  -5.5 -40 -15 
   /
   \setshadesymbol <.1pt,.1pt,.1pt,.1pt> ({\fiverm 3}) \setshadegrid span <4pt> 
    \vshade 
   -5 -40 -12    -3.9 -40 -12   -3.2 -40 -12
   /
   \setshadesymbol <.1pt,.1pt,.1pt,.1pt> ({\fiverm 4}) \setshadegrid span <4pt> 
    \vshade 
  -2.8 -40 -12   -1.5 -40 -12   -0.1 -40 -12 
   /
   \setshadesymbol <.1pt,.1pt,.1pt,.1pt> ({\fiverm 5}) \setshadegrid span <4pt> 
   \vshade 
 0.5 -40 0     1.7 -40 0     3.41 -40 0   
   /
   \setshadesymbol <.1pt,.1pt,.1pt,.1pt> ({\fiverm 6}) \setshadegrid span <4pt> 
   \vshade 
3.5 -40 0  6.24 -40 0  9.15 -40 0 
   /
   \setshadesymbol <.1pt,.1pt,.1pt,.1pt> ({\fiverm 7}) \setshadegrid span <4pt> 
   \vshade 
 9.2 -9.5  0    10 -9.5 0      10.5 -9.5 0 
   /
\setdots \putrule from -9.24 -28.48 to -9.24 0 
\putrule from -7.66 -40 to -7.66 0 
\putrule from -5.41 -40 to -5.41 0 
\putrule from -3.41 -40 to -3.41 0 
\putrule from 0 -40 to 0 -11.7 
\putrule from 9.24 -10 to 9.24 0 
\putrule from -9.24 -10 to 0 -10 
\putrule from -9.24 -11.8 to 0 -11.8 
\putrule from -9.24 -14.97 to 0 -14.97
\endpicture &
\beginpicture
    \setcoordinatesystem units <0.2cm,0.2cm>
    \setplotarea x from -12 to 12, y from -35 to 1
    \put {\tiny{$B^{\#}(i)$}} at 13.5  -15
    \put {\tiny{$B^{\#}(ii)$}} at 13.5  -9
    \put {\tiny{$B^{\#}(iii)$}} at 13.5  -4.5
    \put {\tiny{$B^{\#}(iv)$}} at 13.5  -30
    \put {\tiny{$0=\phi_3$}} at 0.05 0.8
    \put {\tiny{$\phi_0$}} at -9.24 0.8
    \put {\tiny{$\phi_1$}} at -7.66 0.8
    \put {\tiny{$\phi_2$}} at -5.41 0.8
    \put {\tiny{$\frac{\lambda}{2}$}} at 9.24 1.2
    \put {\tiny{$\frac{2}{\lambda}$}} at 10.83 1.2
    \put {\tiny{$-\lambda$}} at -10.6  -20
    \put {\tiny{$-1$}} at -10.6  -10
    \put {\tiny{$-\frac{1}{\lambda}$}} at -10.6  -8
    \putrule from -9.24 -0.5 to -9.24 0.5
    \putrule from -7.66 -0.5 to -7.66 0.5
    \putrule from -5.41 -0.5 to -5.41 0.5
    \putrule from 9.24 -0.5 to 9.24 0.5
    \putrule from -9.24 0 to 10.83 0
    \putrule from -9.24 -40 to -9.24 -28.48
    \putrule from -9.24 -28.48 to -7.66 -28.48
    \putrule from -7.66 -28.48 to -7.66 -14.97
    \putrule from -7.66 -14.97 to -5.41 -14.97
    \putrule from -5.41 -14.97 to -5.41 -11.8
    \putrule from -5.41 -11.8 to 0 -11.8
    \putrule from 0 -11.8 to 0 0.5
    \putrule from 0 -10 to 10.83 -10
    \putrule from 10.83 -10 to 10.83 0.5
    \putrule from 9.24 -40 to 9.24 -10
    \setshadesymbol <.1pt,.1pt,.1pt,.1pt> ({\fiverm 7}) \setshadegrid span <4pt> 
    \vshade  
    -9 -40 -28.6   -8.4 -40 -28.6   -7.8 -40 -28.6 
     /
   \setshadesymbol <.1pt,.1pt,.1pt,.1pt> ({\fiverm 6}) \setshadegrid span <4pt> 
    \vshade 
   -7.6 -40 -21    1 -40 -21  9.15 -40 -21
   /
   \setshadesymbol <.1pt,.1pt,.1pt,.1pt> ({\fiverm 5}) \setshadegrid span <4pt> 
    \vshade 
    0.5  -7.5  0  6  -7.5 0  10.5  -7.5  0
   /
   \setshadesymbol <.1pt,.1pt,.1pt,.1pt> ({\fiverm 4}) \setshadegrid span <4pt> 
   \vshade 
    0.5  -9.5  -8  6  -9.5 -8  10.5  -9.5  -8
   /
   \setshadesymbol <.1pt,.1pt,.1pt,.1pt> ({\fiverm 3}) \setshadegrid span <4pt> 
   \vshade 
   0.5  -20 -10  6 -20 -10   9  -20 -10  
   /
   \setshadesymbol <.1pt,.1pt,.1pt,.1pt> ({\fiverm 2}) \setshadegrid span <4pt> 
   \vshade 
-5.41  -20 -12  -2.5 -20 -12   0  -20 -12
   /
   \setshadesymbol <.1pt,.1pt,.1pt,.1pt> ({\fiverm 1}) \setshadegrid span <4pt> 
   \vshade 
 -7.6 -20 -15  -6.5 -20  -15    -5.7  -20 -15
 /
\setdots 
  \putrule from -9.24  -8 to 10 -8
   \putrule from -9.24 -10 to 0 -10 
   \putrule from -9.24  -20 to 9.24 -20
   \putrule from 10.83 -10 to 15 -10
   \putrule from 9.24 -20 to 15 -20
   \putrule from 10.83 -8 to 15 -8
   \putrule from 10.83 0 to 15 0
\endpicture
$$
\end{tabular}
 \caption[nat-ext]{The region $\Omega^{\ast}$ and its image under $\hat{S}$, here $k=8$.  Note that $\Omega_{0}$ lies below $y = -1\,$.} \label{figOm8}
\end{figure}

\begin{Prop}\label{prop:even region}
The map $\hat{S}$ is surjective from $\Omega^{\ast}$ onto itself
and is injective off of  the boundaries of $J_j \times \bar{K}_j$,
$ 1 \le j \le l+1$. Moreover, $\frac{{\rm d}x\, {\rm d}y}{|x -
y|^2}$ is an invariant measure for $\hat{S}$.
\end{Prop}

\begin{proof} We must simply check that, up to measure zero
niceties, the map $\hat{S}(x,y)$ does indeed act bijectively on
$\Omega^{\ast}$.   This is a matter of elementary calculations,
which we now outline.
\begin{itemize}
\item $x \in [-\frac{\lambda}{2}, \, -\frac{2}{3 \lambda})$;
On this interval, $S(x) = -\frac{1}{x} - \lambda$. Recall that the
$\phi_i$ are in fact the orbit of $\phi_0 = -\frac{\lambda}{2}$
under iteration of this map; in particular, $\phi_{\ell -1}=0$.
Then $\hat{S}$ sends $\cup_{i=1}^{\ell -1}\, [\phi_{i-1}, \,
\phi_{i}) \times [-\infty, \, -\frac{1}{L_i}]$ to
$\cup_{i=1}^{\ell -1} \, [\phi_{i}, \, \phi_{i+1}) \times
[-\lambda_q,\, -1/L_{i+1}\,]$. Also $[\phi_{\ell -1}, \,
-\frac{2}{3 \lambda}) \times [-\infty, -\frac{1}{L_{\ell -1}}]$ is
now sent to $[0, \, \frac{\lambda}{2}) \times [-\lambda, -1]$.

\item $x \in [-\frac{2}{3 \lambda}, \, 0)$;
On this interval, $\hat{S}(x, y) \, = \, \left( \frac{-x}{\lambda
x + 1}, \, \frac{-y}{\lambda y + 1} \right)$ for $(x, y) \in
\Omega^{\ast}$.  Since  $L_{\ell -1} = \lambda-1$, we easily find
that $\hat{S}$ sends $[-\frac{2}{3\lambda}, \, 0) \times [-\infty,
\, -\frac{1}{L_{l-1}}]$  to $[0, \, \frac{2}{\lambda}) \times [-1,
\, -\frac{1}{\lambda}]$.

\item $x \in [0, \, \frac{2}{3 \lambda})$;
On this interval, $\hat{S}(x,y) = \left(\frac{x}{1 - \lambda x},
\, \frac{y}{1 - \lambda y}\right)$ for $(x, y) \in \Omega^{\ast}$.
One immediately finds that $ [0, \, \frac{2}{3 \lambda}) \times
[-\infty, 0]$ is sent to $[0, \, \frac{2}{\lambda}) \times
[-\frac{1}{\lambda}, \, 0]$.

\item $x \in [\frac{2}{3 \lambda} , \, \frac{\lambda}{2})$;
On this interval, $\hat{S}(x,y) = (\frac{1}{x} - \lambda, \,
\frac{1}{y} - \lambda)$ for $(x, y) \in \Omega^{\ast}$. One finds
that $[\frac{2}{3 \lambda} , \, \frac{\lambda}{2}) \times [-\infty
, \, 0)$ is sent to $[\phi_1, \, \frac{\lambda}{2}) \times
[-\infty ,\, -\lambda]$.

\item $x \in [\,\frac{\lambda}{2},\, \frac{2}{\lambda})$;
On this interval, also $\hat{S}(x,y) = (\frac{1}{x} - \lambda, \,
\frac{1}{y} - \lambda)$ for $(x, y) \in \Omega^{\ast}$.  Hence
$[\frac{2}{3 \lambda} , \, \frac{\lambda}{2}) \times [-1 , \, 0]$
is sent to $[-\frac{\lambda}{2} , \, \phi_1) \times [- \infty , \,
-\lambda - 1]$.
\end{itemize}
Consequently, we see that $\hat{S}$ is bijective except for
failing to be injective on the (measure zero) boundaries. The
invariance of the measure holds since $\hat{S}$ is locally of the form  $(x,y) \mapsto (A x, Ay)$ with $A$ a fractional linear transformation. \end{proof}

\noindent {\bf Remark.} Note that
\[
\iint_{\Omega^{\ast}} \frac{{\rm d}x\, {\rm d}y}{|x - y|^2}=\infty \;.
\]
\subsection{Planar system in the odd index case, $k = 2\ell +3$}
We recycle notation, now using $\phi_j$ and $L_j$ as follows (all
necessary calculations are in~\cite{BKS}):
$$
\phi_{0} = - \frac{\lambda}{2},\qquad \text{and}\quad \phi_{j} =
T^{j} \left( - \frac{\lambda}{2} \right),\,\, 0 \le j \le 2\ell
+1.
$$
We recall that
\[
\left\{ \begin{array}{l} -\frac{\lambda}{2} \le \phi_{j} < -
\frac{2}{3 \lambda} \quad \text{for} \quad j \in \{0,1,\ldots
,\ell -1
\} \cup \{\ell +1,\ldots ,2\ell \} \\
\\
-\frac{2}{3 \lambda} < \phi_{\ell} < - \frac{2}{5\lambda} \\
\\
\phi_{2\ell+1} = 0.
\end{array} \right.
\]
Also we put, with $R$ the positive root of $R^2+(2-\lambda )R-1=0$,
\[
\left\{ \begin{array}{ccl}
L_{2\ell} & = &  \lambda  -1/R \\
L_{2\ell +1} & = &  \lambda  - R \\
L_1 &  =  &  \frac{1}{2 \lambda - L_{2\ell}} \\
L_2 &  =  &  \frac{1}{2 \lambda - L_{2\ell +1}} \\
L_{j} & = & \frac{1}{\lambda - L_{j-2}}, \quad 2 < j \le 2\ell +2
\, ,
\end{array} \right.
\]
which are well-defined (see Subsection~3.2 of~\cite{BKS}). Then we
define
\[
\Omega^{\ast} = \bigcup_{j=1}^{2\ell+4} J_j \times \bar{K}_j \, :
\,
\]
where
\[
\left\{ \begin{array}{ccl}
J_{2j} & = &  [\phi_{\ell+j}, \, \phi_j), \quad 1 \le j \le \ell \\
J_{2j-1} & = &  [\phi_{j-1}, \, \phi_{j+l}), \quad 1 \le j \le \ell+1 \\
J_{2\ell+2} & = & [0, \frac{\lambda}{2} ) , \\
J_{2\ell+3} & = & [\frac{\lambda}{2}, 1 )   \\
J_{2\ell+4} & = & [1, \frac{2}{\lambda} ),
\end{array} \right.
\]
$1= -\frac{\phi_{\ell}}{\lambda \phi_{\ell} + 1}$, and
\[
\left\{ \begin{array}{ccl}
\bar{K}_j & = &  [-\infty, \, - \frac{1}{L_j}], \quad 1 \le j \le 2\ell+1 \\
\bar{K}_{2\ell+2} & = & [-\infty , \, 0, ] ,           \\
\bar{K}_{2\ell+3} & = & [-\frac{1}{\lambda - L_{2\ell+1}}, \, 0]
\, = \, [- \frac{1}{R}, \, 0] \\
\bar{K}_{2\ell+4} & = & [-\frac{1}{\lambda - L_{2\ell}}, \, 0] \,
= \, [- R, \, 0] ;
\end{array} \right.
\]
see Figure~\ref{sHatIm9}.

\begin{figure}[h]
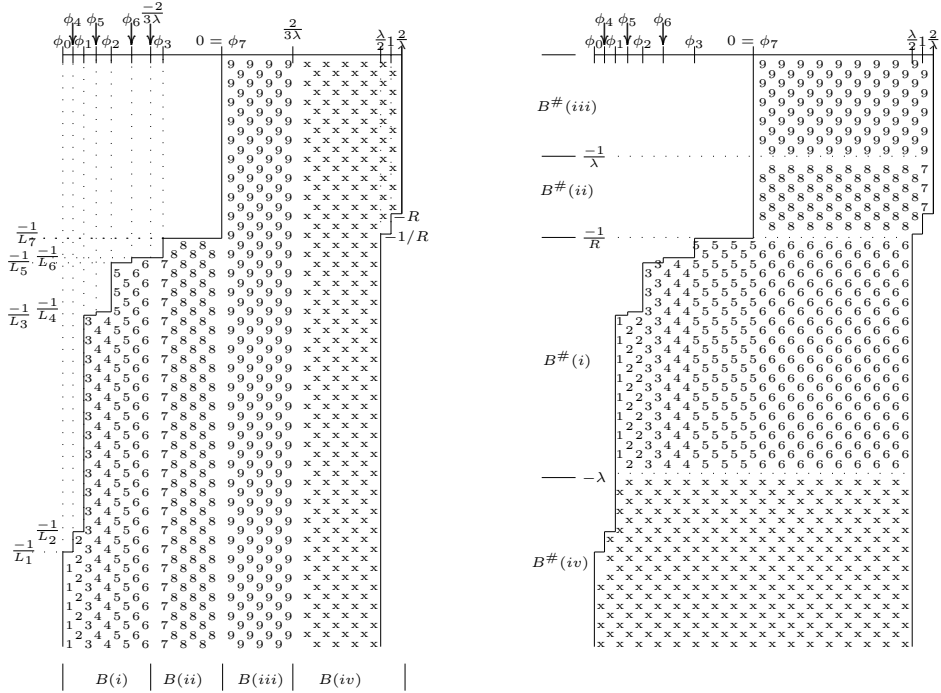

\scalebox{0.9}{
\begin{tabular}{ll}
$$
\beginpicture
    \setcoordinatesystem units <0.25cm,0.25cm>
    \setplotarea x from -15 to 15, y from -35 to 0.5
    \putrule from -9.39  -37.6 to -9.39 -36
     \put {\tiny{$B(i)$}} at -6.5  -37
    \putrule from -4.2  -37.6 to  -4.2 -36
    \put {\tiny{$B(ii)$}} at -2.3   -37
    \putrule from 0.05  -37.6 to  0.05 -36
    \put {\tiny{$B(iii)$}} at 2.3   -37
    \putrule from 4.2  -37.6 to  4.2 -36
    \put {\tiny{$B(iv)$}} at 7   -37
    \putrule from 10.83 -37.6 to  10.83 -36
    \put {\tiny{$0=\phi_7$}} at 0 0.8
    \put {\tiny{$\phi_0$}} at -9.39 0.8
    \put {\tiny{$\phi_1$}} at -8.15 0.8
    \put {\tiny{$\phi_2$}} at -6.53 0.8
    \put {\tiny{$\phi_3$}} at -3.47 0.8
    \put {\tiny{$\phi_4$}} at -8.79 2.1
    \arrow<4pt> [.2,.8] from -8.79 1.85 to -8.79 0.6
    \put {\tiny{$\phi_5$}} at -7.42 2.1
    \arrow<4pt> [.2,.8] from -7.42 1.85 to -7.42 0.6
    \put {\tiny{$\phi_6$}} at -5.32 2.1
    \arrow<4pt> [.2,.8] from -5.32 1.85 to -5.32 0.6
    \put {\tiny{$\frac{-2}{3 \lambda}$}} at -4.1 2.5
    \arrow<4pt> [.2,.8] from -4.2 1.8 to -4.2 0.6
     \put {\tiny{$\frac{2}{3 \lambda}$}} at 4.2 1.5
    \put {\tiny{$\frac{\lambda}{2}$}} at 9.39 0.9
    \put {\tiny{$\frac{2}{\lambda}$}} at 10.64 0.9
    \put {\tiny{$\frac{-1}{L_1}$}} at -11.7  -29.4
    \put {\tiny{$\frac{-1}{L_2}$}} at -10.3  -28.2
    \put {\tiny{$\frac{-1}{L_3}$}} at -12  -15.45
    \put {\tiny{$\frac{-1}{L_4}$}} at -10.3  -15.1
    \put {\tiny{$\frac{-1}{L_5}$}} at -12  -12.3
    \put {\tiny{$\frac{-1}{L_6}$}} at -10.3  -12
    \put {\tiny{$\frac{- 1}{L_7}$}} at -11.5  -10.6
    \put {\tiny{$-1/R$}} at 10.9  -10.8
    \put {\tiny{$-R$}} at 10.9  -9.6
    \put {\tiny{$1$}} at 10  0.8
    \putrule from 10 -0.5 to 10 0.5
    \putrule from -9.39 0 to 10.64 0
    \putrule from -9.39 -35 to -9.39 -29.4
    \putrule from -9.39 -0.5 to -9.39 0.5
    \putrule from -9.39 -29.4 to -8.79 -29.4
    \putrule from -8.79 -29.4 to -8.79 -28.2
    \putrule from -8.79 -0.5 to -8.79 0.5
    \putrule from -8.79 -28.2 to -8.15 -28.2
    \putrule from -8.15 -28.2 to -8.15 -15.4
    \putrule from -8.15 -0.5 to -8.15 0.5
    \putrule from -8.15 -15.4 to -7.42 -15.4
    \putrule from -7.42 -0.5 to -7.42 0.5
    \putrule from -7.42 -15.4 to -7.42 -15.2
    \putrule from -7.42 -15.2 to -6.53 -15.2
    \putrule from -6.53 -15.2 to -6.53 -12.3
    \putrule from -6.53 -0.5 to -6.53 0.5
    \putrule from -6.53 -12.3 to -5.32 -12.3
    \putrule from -5.32 -12.3 to -5.32 -12
    \putrule from -5.32 -0.5 to -5.32 0.5
    \putrule from -5.32 -12 to -3.47 -12
    \putrule from -4.2 -0.5 to -4.2 0.5
    \putrule from -3.47 -0.5 to -3.47 0.5
    \putrule from -3.47 -12 to -3.47 -10.85
    \putrule from -3.47 -10.85 to 0 -10.85
    \putrule from 0 -10.85 to 0 0.5
    \putrule from 9.39 -10.6 to 10 -10.6
    \putrule from  4.2 -0.5 to 4.2 0.5
    \putrule from 10.64 -9.4 to 10.64 0.5
    \putrule from 10 -10.6 to 10 -9.4
    \putrule from 10 -9.4 to 10.64 -9.4
    \putrule from 10.64 -9.4 to 10.64 0.5
    \putrule from 9.39 -35 to 9.39 -10.6
    \putrule from 9.39 -0.5 to 9.39 0.5
%
    \setshadesymbol <.1pt,.1pt,.1pt,.1pt> ({\fiverm 1}) \setshadegrid span <4pt> 
    \vshade  
    -9.39 -35 -29.4   -8.9 -35 -29.4   -8.79 -35 -29.4 
     /
    \setshadesymbol <.1pt,.1pt,.1pt,.1pt> ({\fiverm 2}) \setshadegrid span <4pt> 
    \vshade  
    -8.79 -35 -28.2  -8.45 -35 -28.2   -8.15 -35 -28.2  
     /
    \setshadesymbol <.1pt,.1pt,.1pt,.1pt> ({\fiverm 3}) \setshadegrid span <4pt> 
    \vshade  
     -8.15 -35 -15.45  -7.79 -35 -15.45  -7.42 -35 -15.45    
     /
    \setshadesymbol <.1pt,.1pt,.1pt,.1pt> ({\fiverm 4}) \setshadegrid span <4pt> 
    \vshade  
     -7.42 -35 -15.2  -7.15 -35 -15.2  -6.53 -35 -15.2     
     /
    \setshadesymbol <.1pt,.1pt,.1pt,.1pt> ({\fiverm 5}) \setshadegrid span <4pt> 
    \vshade  
     -6.53 -35 -12.4   -6 -35 -12.4  -5.5 -35 -12.4    
     /
    \setshadesymbol <.1pt,.1pt,.1pt,.1pt> ({\fiverm 6}) \setshadegrid span <4pt> 
    \vshade  
      -5.1 -35 -12.1 -4.7 -35 -12.1   -4.2 -35 -12.1    
     /
    \setshadesymbol <.1pt,.1pt,.1pt,.1pt> ({\fiverm 7}) \setshadegrid span <4pt> 
    \vshade  
     -3.9 -35 -12.1     -3.45 -35 -12.1 -3.3  -35 -12.1   
     /
    \setshadesymbol <.1pt,.1pt,.1pt,.1pt> ({\fiverm 8}) \setshadegrid span <4pt> 
    \vshade  
    -3.3 -35 -10.85  -1.6 -35 -10.85  0 -35 -10.85
     /
    \setshadesymbol <.1pt,.1pt,.1pt,.1pt> ({\fiverm 9}) \setshadegrid span <4pt> 
    \vshade  
    0 -35 0  2.1 -35 0  4.2 -35 0
     /
    \setshadesymbol <.1pt,.1pt,.1pt,.1pt> ({\fiverm x}) \setshadegrid span <4pt> 
    \vshade  
  4.5 -35 0    6.7 -35 0  9.39 -35 0  
  /
    \setshadesymbol <.1pt,.1pt,.1pt,.1pt> ({\fiverm x}) \setshadegrid span <4pt> 
    \vshade  
  9.39 -10.6 0   9.7 -10.6 0  10 -10.6 0
     /
    \setshadesymbol <.1pt,.1pt,.1pt,.1pt> ({\fiverm x}) \setshadegrid span <4pt> 
    \vshade  
10 -9.6 0    10.3 -9.6 0   10.64 -9.6 0  
     /

\setdots 
\putrule from -9.39 -29.4 to -9.39 0 
\putrule from -8.79 -29.4 to -8.79 0 
\putrule from -8.15 -28.2 to -8.15 0 
\putrule from -7.42 -15.1 to -7.42 0 
\putrule from -6.53 -12.3 to -6.53 0 
\putrule from -5.32 -12 to -5.32 0 
\putrule from -4.2 -12 to -4.2 0 
\putrule from -3.47 -10.85 to -3.47 0
\putrule from 9.39 -10.6 to 9.39 0
\putrule from 10 -10.6 to 10 0
\putrule from -11.2  -10.85 to -3.47  -10.85
\putrule from -11.2  -29.4 to -9.24 -29.4
\putrule from -11.2  -12.3 to -6.53 -12.3
\putrule from -10   -12 to   -5.32 -12 
\arrow<4pt> [.2,.8] from -11.2  -10.85 to -3.47  -10.85
\endpicture &
\beginpicture
    \setcoordinatesystem units <0.25cm,0.25cm>
    \setplotarea x from -15 to 15, y from -35 to 0.5
    \put {\tiny{$B^{\#}(i)$}} at -11  -18
    \put {\tiny{$B^{\#}(ii)$}} at -11  -8
    \put {\tiny{$B^{\#}(iii)$}} at -11  -3
    \put {\tiny{$B^{\#}(iv)$}} at -11.4  -30
    \put {\tiny{$0=\phi_7$}} at 0 0.8
    \put {\tiny{$\phi_0$}} at -9.39 0.8
    \put {\tiny{$\phi_1$}} at -8.15 0.8
    \put {\tiny{$\phi_2$}} at -6.53 0.8
    \put {\tiny{$\phi_3$}} at -3.47 0.8
    \put {\tiny{$\phi_4$}} at -8.79 2.1
    \arrow<4pt> [.2,.8] from -8.79 1.85 to -8.79 0.6
    \put {\tiny{$\phi_5$}} at -7.42 2.1
    \arrow<4pt> [.2,.8] from -7.42 1.85 to -7.42 0.6
    \put {\tiny{$\phi_6$}} at -5.32 2.1
    \arrow<4pt> [.2,.8] from -5.32 1.85 to -5.32 0.6
    \put {\tiny{$\frac{\lambda}{2}$}} at 9.39 0.9
    \put {\tiny{$\frac{2}{\lambda}$}} at 10.64 0.9
    \put {\tiny{$\frac{-1}{R}$}} at -9.4  -10.8
     \put {\tiny{$-\lambda$}} at -9.4  -25
    \put {\tiny{$\frac{-1}{\lambda}$}} at -9.4  -6
    \put {\tiny{$1$}} at 10  0.8
    \putrule from 10 -0.5 to 10 0.5
    \putrule from -9.39 0 to 10.64 0
    \putrule from -9.39 -35 to -9.39 -29.4
    \putrule from -9.39 -0.5 to -9.39 0.5
    \putrule from -9.39 -29.4 to -8.79 -29.4
    \putrule from -8.79 -29.4 to -8.79 -28.2
    \putrule from -8.79 -0.5 to -8.79 0.5
    \putrule from -8.79 -28.2 to -8.15 -28.2
    \putrule from -8.15 -28.2 to -8.15 -15.4
    \putrule from -8.15 -0.5 to -8.15 0.5
    \putrule from -8.15 -15.4 to -7.42 -15.4
    \putrule from -7.42 -0.5 to -7.42 0.5
    \putrule from -7.42 -15.4 to -7.42 -15.2
    \putrule from -7.42 -15.2 to -6.53 -15.2
    \putrule from -6.53 -15.2 to -6.53 -12.3
    \putrule from -6.53 -0.5 to -6.53 0.5
    \putrule from -6.53 -12.3 to -5.32 -12.3
    \putrule from -5.32 -12.3 to -5.32 -12
    \putrule from -5.32 -0.5 to -5.32 0.5
    \putrule from -5.32 -12 to -3.47 -12
    \putrule from -3.47 -0.5 to -3.47 0.5
    \putrule from -3.47 -12 to -3.47 -10.85
    \putrule from -3.47 -10.85 to 0 -10.85
    \putrule from 0 -10.85 to 0 0.5
    \putrule from 9.39 -10.6 to 10 -10.6
    \putrule from 10.64 -9.4 to 10.64 0.5
    \putrule from 10 -10.6 to 10 -9.4
    \putrule from 10 -9.4 to 10.64 -9.4
    \putrule from 10.64 -9.4 to 10.64 0.5
    \putrule from 9.39 -35 to 9.39 -10.6
    \putrule from 9.39 -0.5 to 9.39 0.5
\putrule from -12.5 0 to -10.5 0
\putrule from -12.5 -6 to -10.5 -6
\putrule from -12.5 -10.8 to -10.5 -10.8
\putrule from -12.5 -25 to -10.5 -25
%
    \setshadesymbol <.1pt,.1pt,.1pt,.1pt> ({\fiverm x}) \setshadegrid span <4pt> 
    \vshade  
 -9.39 -35 -32
 -8.79 -35  -30
 -8.6 -35  -30
-8.15 -35  -25  
9.39      -35 -25   
     /
    \setshadesymbol <.1pt,.1pt,.1pt,.1pt> ({\fiverm 9}) \setshadegrid span <4pt> 
    \vshade  
    0  -6 0   5.3  -6 0   10.64  -6 0 
     /
    \setshadesymbol <.1pt,.1pt,.1pt,.1pt> ({\fiverm 8}) \setshadegrid span <4pt> 
    \vshade  
    0  -10.6 -6.5   5.3   -10.6 -6.5   10 -10.6 -6.5 
     /
    \setshadesymbol <.1pt,.1pt,.1pt,.1pt> ({\fiverm 7}) \setshadegrid span <4pt> 
    \vshade  
    10   -9.6 -6.5   10.3   -9.6 -6.5   10.64   -9.6 -6.5 
    /  
    \setshadesymbol <.1pt,.1pt,.1pt,.1pt> ({\fiverm 6}) \setshadegrid span <4pt> 
    \vshade  
    0   -24.5 -10.8    5   -24.5 -10.8   9.5   -24.5 -10.8 
    /  
    \setshadesymbol <.1pt,.1pt,.1pt,.1pt> ({\fiverm 5}) \setshadegrid span <4pt> 
    \vshade  
      -3.47   -24.5 -11.1   -1.6   -24.5 -11.1   0.1   -24.5 -11.1
    / 
    \setshadesymbol <.1pt,.1pt,.1pt,.1pt> ({\fiverm 4}) \setshadegrid span <4pt> 
    \vshade  
        -5.32   -24.5 -12   -4.4   -24.5 -12  -3.47   -24.5 -12 
    /  
    \setshadesymbol <.1pt,.1pt,.1pt,.1pt> ({\fiverm 3}) \setshadegrid span <4pt> 
    \vshade  
    -6.53   -24.5 -12.3  -6.0   -24.5 -12.3  -5.32   -24.5 -12.3
    /  
    \setshadesymbol <.1pt,.1pt,.1pt,.1pt> ({\fiverm 2}) \setshadegrid span <4pt> 
    \vshade  
 -7.42   -24.5 -15.3  -6.32   -24.5 -15.3     -6.53   -24.5 -15.3
    /  
    \setshadesymbol <.1pt,.1pt,.1pt,.1pt> ({\fiverm 1}) \setshadegrid span <4pt> 
    \vshade  
-8.15   -24.5 -15.45     -7.73   -24.5 -15.45   -7.42   -24.5 -15.45
    /  
\setdots 
\putrule from -8 -6 to 10 -6
\putrule from  -8   -10.8 to -3.47 -10.8
\putrule  from 0  -10.8  to 9 -10.8
\putrule from  -8  -24.75 to  9.4  -24.75
\endpicture
$$ 
\end{tabular}
}
\caption{The region $\Omega^{\ast}$ and its image under $\hat{S}$, here $k=9$.  Here,  $\Omega_{0}$ lies below $y = -1/R\,$.} \label{sHatIm9}
\end{figure}

\begin{Prop}\label{prop:odd region}
The map $\hat{S}(x,y)=\left( M_1^{-1}(x), \, M_1^{-1}(y) \right)$
of $\Omega^{\ast}$ is bijective off of the boundaries of $J_j
\times \bar{K}_j$, $1 \le j \le 2l+2$.  Moreover, $\frac{{\rm d}x\, {\rm
d}y}{|x - y|^{2}}$ is an invariant measure for $\hat{S}$.
\end{Prop}

\begin{proof} The invariance of the measure has already been remarked upon.  
The first part of the assertion follows as in Proposition ~\ref{prop:even region}:
\begin{itemize}
\item $x \in [-\frac{\lambda}{2}, \, -\frac{2}{3 \lambda})$;
Here, $\hat{S}(x, y) \, = \, \left( -\frac{1}{x} - \lambda, \,
-\frac{1}{y} - \lambda \right)$. Thus, the corresponding image of
$\Omega^{\ast}$ is fibred below $[\phi_1, \, \frac{\lambda}{2})$,
with $y$ values in $[-\lambda, \, -\frac{1}{L_j}]$ for $S(x)$
negative, with the appropriate value of $L_j$, and $y \in
[-\lambda , \, -\frac{1}{R}]$ for $S(x)$ non-negative.

\item $x \in [-\frac{2}{3 \lambda}, \, 0)$;
Here $\hat{S}(x, y)=\left( \frac{-x}{\lambda x + 1}, \,
\frac{-y}{\lambda y + 1} \right)$ for $(x, y) \in \Omega^{\ast}$.
Thus $[-\frac{2}{3 \lambda}, \, 0) \times [-\infty, \,
-\frac{1}{L_{2\ell +1}}]$  is sent to $[0, \, \frac{2}{\lambda})
\times [-\frac{1}{\lambda - L_{2\ell +1}}, \,
-\frac{1}{\lambda}]$.

\item $x \in [0, \, \frac{2}{3 \lambda})$;
For these values of $x$, $\hat{S}(x,y) = \left(\frac{x}{1 -
\lambda x}, \, \frac{y}{1 - \lambda y}\right)$ for $(x, y) \in
\Omega^{\ast}$. Thus $ [0, \, \frac{2}{3 \lambda}) \times
[-\infty, 0]$ is sent to $[0, \, \frac{2}{\lambda}) \times
[-\frac{1}{\lambda}, \, 0]$.

\item $x \in [\frac{2}{3 \lambda} , \, \frac{\lambda}{2})$;
Here $\hat{S}(x,y) = (\frac{1}{x} - \lambda, \, \frac{1}{y} -
\lambda)$ for $(x, y) \in \Omega^{\ast}$. Thus $[\frac{2}{3
\lambda} , \, \frac{\lambda}{2}) \times [-\infty , \, 0)$ is sent
to $[\phi_1, \, \frac{\lambda}{2}) \times [-\infty ,\, -\lambda]$.

\item $x \in [\frac{\lambda}{2},\, \alpha)$;
Here again  $\hat{S}(x,y) = (\frac{1}{x} - \lambda, \, \frac{1}{y}
- \lambda)$ for $(x, y) \in \Omega^{\ast}$.  Then $[\frac{2}{3
\lambda} , \, \alpha) \times [-\frac{1}{R} , \, 0]$ is sent to
$[\phi_{\ell+1} , \, \phi_1) \times [- \infty , \, -R - \lambda ]$
\, = \, $[\phi_{\ell+1} , \, \phi_1) \times [- \infty , \,
-\frac{1}{L_2}]$.

\item $x \in [\alpha, \, \frac{2}{\lambda})$;
Once again, $\hat{S}(x,y) = (\frac{1}{x} - \lambda, \, \frac{1}{y}
- \lambda)$ for $(x, y) \in \Omega^{\ast}$ and $[\alpha , \,
\frac{2}{\lambda}) \times [-\frac{1}{\lambda - L_{2\ell}} , \, 0]$
is sent to $[-\frac{\lambda}{2} , \, \phi_{\ell+1}) \times [-
\infty , \, -\lambda - \frac{1}{L_1}]$.
\end{itemize}
Combining the above, we get the first part of the assertion of the
proposition.
\end{proof}

\subsection{Proof of Theorem \ref{thm:natural-extension}:  Planar System is Natural Extension }\label{fritzToRescue}
We first recall some terminology  and notation
from~\cite{Schw}.    Let $B$ be a set, and $T: B\to B$ be a map.
The pair $(B,T)$ is called a \emph{fibred system} if the following
three conditions are satisfied:
\begin{itemize}
\item[(a)] There is a finite or countable set $I$ (called the
digit set);
\item[(b)] There is a map $k: B\to I$. Then the sets
$B(i)=k^{-1}(i)$ form a partition of $B$;
\item[(c)] The restriction of $T$ to any $B(i)$ is an injective
map;
\end{itemize}
see~\cite{Schw}, Definition~1.1.1.
From  our definition of the Rosen mediant map, see Definition ~\ref{def:RosMed}, 
 we naturally have
$B={\mathbb J}_k=[-\lambda /2, 2/\lambda ]$, $T=S$,
$I=\{ i, ii, iii, iv\}$, and
$B(i)=[-\frac{\lambda}{2}, -\frac{2}{3\lambda})$, $B(ii)=[-\frac{2}{3\lambda}, 0)$,
$B(iii)=[0,\frac{2}{3\lambda})$, and $B(iv)=[\frac{2}{3\lambda},\frac{2}{\lambda})$;
see also Figures~\ref{figOm8} and ~\ref{sHatIm9}.\smallskip\

The pair $(B^{\#},T^{\#})$ is called a \emph{dual fibred system}
(or \emph{backward algorithm}) with respect to $(B,T)$ if the
following condition holds: $(k_1,k_2,\dots,k_n)$ is an admissible
block of digits for $T$ if and only if $(k_n,\dots, k_2,k_1)$ is
admissible for $T^{\#}$; see Definition~21.1.1 in~\cite{Schw}.
Furthermore,
\begin{eqnarray*}
D(x)&:=&\{ y\in B^{\#}\,\vert\, y\in B^{\#}(k_1,k_2,\dots,k_N)
\text{ if and only if }\\
& & \quad T^{-N}(x)\cap B(k_N,\dots,k_2,k_1)\neq \emptyset ,\,
\text{ for all $N\geq 1$}\} ;
\end{eqnarray*}
see Definition~21.1.7 from~\cite{Schw}. The local inverse of the
map $T: B(k)\to B$ is denoted by $V(k)$. Schweiger obtained the
following theorem; see~\cite{Schw}, Theorem~21.2.1.      

\begin{Thm}[Schweiger]  Consider the following dynamical system
$(\bar{B},\bar{T})$, with $\bar{B}=\{ (x,y); x\in B, y\in D(x)\}$,
and where $\bar{T}:\bar{B}\to \bar{B}$ is defined by
\[
\bar{T}(x,y)=\left( \,T(x), V^{\#}(k(x))(y)\, \right)\, .
\]
If $\bar{T}$ is measurable, then $(\bar{B},\bar{T})$ equipped with
the obvious product $\sigma$-algebra is an invertible dynamical
system.  Furthermore, if $K$ is a non-negative measurable function such that 
\[K(Tx, y) \, \vert T'(x)\,\vert = K(x, T^{\#}y) \, \vert\, (T^{\#})' y\,\vert\]
then $K$ is an invariant density for this system.
The dynamical system
$(\bar{B},\bar{T})$ is the natural extension of $(B,T)$.
\end{Thm}

In our setting we have by construction that
$\bar{B}=\Omega^*$. So to apply Schweiger's Theorem, we
need to find a backward algorithm $(B^{\#},T^{\#})$, such that
\begin{equation}\label{Vshap-map}
V^{\#}(k(x))(y)=M_1^{-1}(y).
\end{equation}
Thus,  on each $B^{\#}(\kappa)$   we have that $T^{\#}$ is  given by the inverse of the matrix giving  the Rosen mediant map on the corresponding $B(\kappa)$.  Comparing with Definition ~\ref{def:RosMed},   this 
map must be
$$
T^{\#}(y)=\left\{ \begin{array}{ll}
\dfrac{-1}{y+\lambda},  & y\in B^{\#}(i)\,;\\
 & \\
\dfrac{-y}{\lambda y+1}, & y\in B^{\#}(ii)\,;\\
 & \\
\dfrac{y}{\lambda y+1}, & y\in B^{\#}(iii)\,;\\
 & \\
\dfrac{1}{y+\lambda},  & y\in B^{\#}(iv)\,.
\end{array}\right.
$$

Using the proofs of Propositions ~\ref{prop:even region} and ~\ref{prop:odd region},   we solve to find the partition of $B^{\#}= (-\infty, 0]$   (recall that $R=1$ in the even index case): 
$$
B^{\#}(i)=[-\lambda , -1/R),\quad B^{\#}(ii)=[-1/R,-\tfrac{1}{\lambda}),
\quad B^{\#}(iii)=[-\tfrac{1}{\lambda},0),
$$
and
$B^{\#}(iv)=(-\infty, -\lambda )$; again, see Figures~\ref{figOm8} and ~\ref{sHatIm9}.
An easy calculation shows that $V^{\#}(k(y))$
satisfies~(\ref{Vshap-map}).\smallskip

We find that $(B^{\#},T^{\#})$ is the dual fibred system with
respect to $({\mathbb J}_k,S)$. We already saw that an
invariant measure is given by
$\frac{{\rm d}x\, {\rm d}y}{|x -
y|^2}$ is an invariant measure for $\hat{S}$, note that this is compatible with the requirements on $K$ in Schweiger's theorem. Thus it follows that $(\Omega^{*},
\bar{\mathcal B}, \nu,{\hat S})$ is the
natural extension of the dynamical system $({\mathbb J}_k,
{\mathcal B}, \mu_k,S)$.
\qed

\bigskip 

\noindent {\bf Remark.} Note  that one could use the above method to verify that the system given in \cite{BKS} is indeed the natural extension of the Rosen map.   With this, as suggested by the referee, one can use the relationship between $S$ and $T$ to show that $\hat{S}$ is indeed the map giving the natural extension of $S$.

\subsection{Ergodicity}\label{Ergodicity}
We denote by $\hat{\mu}$ the measure defined by $\frac{1}{(x - y)^2}$ as its
density function with respect to Lebesgue measure and by $\mu$ its marginal distribution on the first
coordinate.
\begin{Thm}\label{thm:ErgodicNatEx}
The dynamical system $(\Omega^{\ast}, \, \hat{S} , \, \hat{\mu})$
is ergodic, and its entropy $h(\hat{S}, \, \hat{\mu})$ is equal to
$\frac{(k-2)\pi^2}{2k}$.
\end{Thm}
\begin{proof}
An easy calculation shows that
\[
\hat{S}^{k_m} (x, y) \, = \, \left( (M_1 \cdots M_{k_m} )^{-1}
(x), \, (M_1 \cdots M_{k_m} )^{-1} (y) \right)
\]
and
\[
(M_1 \cdots M_{k_m} )^{-1}(x) =  T^{m}(x) .
\]
 Thus,  $\hat{S}^{k_m}$ is the induced transformation
$\hat{S}_{\Omega_0}$ of $\hat{S}$ to $\Omega_0$, and is
conjugate to ${\mathcal T}$ by the isomorphism $(x, y) \to (x,
-1/y)$.  Since ${\mathcal T}$ is ergodic (again, see ~\cite{BKS}), so
is $\hat{S}_{\Omega_0}$. In turn, this implies the ergodicity of
$\hat{S}$.\smallskip\

The entropy $h(\hat{S}, \, \hat{\mu})$ of $(\hat{S}, \,
\hat{\mu})$ is given by the entropy of its induced transformation
on the region $\Omega_0$, as
$$
h(\hat{S}, \, \hat{\mu})=h(\hat{S}_{\Omega_0},\,
\hat{\mu}_{\Omega_0})\cdot \hat{\mu}(\Omega_0),
$$
where $\hat{\mu}_{\Omega_0}$ is the restricted normalized measure
of $\hat{\mu}$ to $\Omega_0$; see~\cite{krengel}. Since
$(\hat{S}_{\Omega_0},\, \hat{\mu}_{\Omega_0})$ is a natural
extension of the Rosen map, and its entropy is
$$
C\cdot \frac{(k-2)\pi^2}{2k},
$$
where $C$ is the normalizing constant of the invariant measure, i.e.,
$$
C^{-1}=\iint_{\Omega_0} \frac{{\rm d}x\, {\rm d}y}{|x-y|^2},
$$
(see~\cite{Na4}), the result follows.
\end{proof}

The following result is an immediate consequence of
Theorem~\ref{thm:ErgodicNatEx}.
\begin{Cor}\label{cor:ErgodicityOfTheMap}
The dynamical system $(S, \mu)$ is ergodic, and its entropy
$h(S,\, \mu)$ is equal to $\frac{(k-2)\pi^2}{2k}$.
\end{Cor}

\noindent \textbf{Remark.} The density function with respect to Lebesgue measure $f$ of the measure
$\mu$ is given by
\[
f(x) \, = \, \int_{\{y \, : \, (x,y) \in \Omega^{\ast} \} }
\frac{{\rm d}y}{(x - y)^2}
\]
which diverges at $ x = 0$.

\begin{Cor}
For a.e. $x \in {\mathbb I}_k$, $\left\{ v_{m,l} \left| x \, - \,
\frac{u_{m,l}}{v_{m,l}}\right| \, : \,  1 \le l \le r_m - 1, m \ge
1 \right\} $ is unbounded.
\end{Cor}

\begin{proof} From the ergodicity of $\hat{S}$, for a.e. $(x,y) \in
\Omega^{\ast}$, its forward $\hat{S}$-orbit is dense in
$\Omega^{\ast}$.  It follows that distance between the second
coordinates of $\hat{S}^n (x, y)$ and $\hat{S}^n (x, -\infty)$
tends to $0$. Thus we see that the forward $\hat{S}$-orbit of $(x,
-\infty)$ is also dense in $\Omega^{\ast}$. By Fubini's theorem,
this holds for a.e. $x \in {\mathbb I}_k$. Because
\[
\hat{S}^{n} (x, -\infty) \, = \, \left(S^{n} (x) , -
\frac{q_m}{v_{m,s}} \right)
\]
for $n \, = \, k_{m} + s$ and by (4), we have the assertion.
\end{proof}
\section{Proof of Theorem \ref{witness} --- exhibiting a witness}\label{sec:hurwitz}
In this section,  we  use continued fraction methods to prove Theorem \ref{Haas+Series}.    That is, we display the {\em witness} mentioned in the Introduction.      For each $k>3$,  we call our witness $\tau_{0}$; this value is suggested by the geometry of our planar natural extension.   
Let
\begin{equation}\label{approxCoeffs}
\Theta \left( x, \frac{a}{c}\right) =c^2 \left| \, x \, - \,
\frac{a}{c} \right|
\end{equation}
for $\begin{pmatrix} a & \cdot \\ c & \cdot \end{pmatrix} \in
G_{k}$. Hereafter, when we write $a/c$ we assume that there exists
a matrix $\begin{pmatrix} a & \cdot \\ c & \cdot
\end{pmatrix} \in G_{k}$.   In case $a/c$ is equal to the $n$th mediant 
convergent of $x$ --- with indexing as for Equation \eqref{eq:distance-x-to-mediant} ---, we write $\Theta_n(x)$
instead of $\Theta \left( x,a/c\right)$. 

The approach in both the even and odd index cases is quite similar, the odd case is --as usual with the $G_k$-- a little bit more complicated.  We give full detail of the even case, and outline the odd case.

\subsection{Even index case: $k = 2\ell$}
One of the main ingredients of the aforementioned Borel-type result
from~\cite{KSS} is the fact that if we set
$$
\tau_0:=[\, \overline{(-1:1)^{\ell-2},\, (-1:2)}\, ]=1-\lambda =-
L_{\ell -1},\quad \eta_0:=L_1=\frac{1}{\lambda +1}
$$
(here the bar indicates periodicity), the sequence
$$
(\tau_i,\eta_i):= {\mathcal T}^i(\tau_0,\eta_0),\quad \text{for
$i\geq 0$},
$$
is purely periodic, with period-length $\ell-1$. Here ${\mathcal
T}:\Omega \to \Omega$ is the natural extension map
from~\cite{BKS}, see Equation \eqref{eq:BKSextension}.
Note that $\tau_{\ell-2}=\frac{-1}{\lambda +1}=-L_1$,
$\eta_{\ell-2}=\lambda -1=L_{\ell-1}$, and that ${\mathcal
T}(\tau_{\ell-2},\eta_{\ell-2})=(\tau_0,\eta_0)$. Furthermore,
in~\cite{BKS} it was shown that if
$$
(t_n,v_n)={\mathcal T}^n(x,0),\quad \text{for $x\in [-\lambda
/2,\lambda /2)$, and $n\geq 0$},
$$
one has that
$$
\theta_{n-1}(x)=\frac{v_n}{1+t_nv_n},\qquad \text{and}\quad
\theta_n (x)=\frac{|t_n|}{1+t_nv_n},\qquad n\geq 1.
$$
Due to this
$$
\theta(\tau_i,\eta_i):=\frac{\eta_i}{1+\tau_i\eta_i}\geq
C(k)=\frac{1}{2} $$ for $i\geq 0$ (with equality if $i\equiv 0\,\,
(\text{mod } \ell-1)$ or $i\equiv \ell-2\,\, (\text{mod }
\ell-1)$), and since $\theta_{i-1+n(\ell-1)}(\tau_0)\uparrow
\theta (\tau_i,\eta_i)$, as $n\to\infty$, it follows that for any
$C<C(k)=1/2$,
$$
\theta_n (\tau_0)<C,\quad \text{for at most finitely many $n\geq
0$}.
$$

Recall that $\Omega_0$ is an isomorphic copy of $\Omega$; for
$i=0,1,\dots,\ell-2$, the points $(\tau_i,\eta_i)\in\Omega$
correspond to the points $(\tau_i,K_{i+1})\in\Omega_0$. Since 
the isomorphic copy of the system $(\Omega_0, \mathcal T)$ is induced from 
$(\Omega^{\ast},\hat{S})$, the $\hat{S}$-orbit of $(\tau_0, K_{1})$   has cardinality at least $\ell -1$.     In fact, this
$\hat{S}$-orbit is also purely periodic, but is of cardinality $\ell$:   since $\tau_i<
-2/3\lambda$, for $i=0,1,\dots,\ell-3$, and
$\tau_{\ell-2}>-2/3\lambda$, we have that
$$
\hat{S}(\tau_{\ell-2}, K_{\ell-1})=\hat{S}\left( \frac{-1}{\lambda
+1}, \frac{1}{\lambda -1}\right) =(1,-1),
$$
and
$$
\hat{S}(1,-1)=\left( \frac{1}{1}-\lambda , \frac{1}{-1}-\lambda
\right) =(1-\lambda , -1-\lambda) = (\tau_0, K_1)\;.
$$

  Setting
$$
(T_n,V_n)=\hat{S}^n(\tau_0,-\infty ),\qquad n\geq 0,
$$
it follows from~(\ref{distance}), and the fact that (an isomorphic
copy of) ${\mathcal T}$ is an induced transformation of $\hat{S}$,
that $(\theta_n(\tau_0))_{n\geq 0}$ is a subsequence of
$(\Theta_n(\tau_0))_{n\geq 0}$. In fact the only points in the
latter sequence which are not in the former are among the numbers
$\Theta (1,y)$, with $-1<y<0$. For these numbers we have that
$\Theta (1,y)=\frac{1}{1-y}>\Theta (1,-1)=\frac{1}{2}$.
Consequently, we find for any $C<C(k)=1/2$,
$$
\Theta_n (\tau_0)<C,\quad \text{for at most finitely many $n\geq
0$}.
$$
\subsection{Odd index case: $k = 2\ell +3$}
Analogous to the even index case, we set $\tau_0$ equal to the
``left-top height'' of $\Omega$, and $\eta_0$ equal to the lowest
``height'' of $\Omega$, i.e.,
$$
\tau_0:=- L_{2\ell +1}=R-\lambda,\quad
\eta_0:=L_1=\frac{1}{\lambda +\frac{1}{R}},
$$
and we set
$$
(\tau_i,\eta_i):= {\mathcal T}^i(\tau_0,\eta_0),\quad \text{for
$i\geq 0$}.
$$
Again the sequence $(\tau_i,\eta_i)_{i\geq 0}$ is purely periodic,
with period-length $2\ell$; see~\cite{KSS}. Contrary to the even
case, this sequence is more ``complicated,'' with a kind of
``double loop.'' On $\Omega_0$ the sequence corresponding with
$(\tau_i,\eta_i)_{i\geq 0}$ is the (purely periodic) sequence
$(\tau_i,K_{i+1})_{i\geq 0}$; see Figure~\ref{fig3}.
\begin{figure}
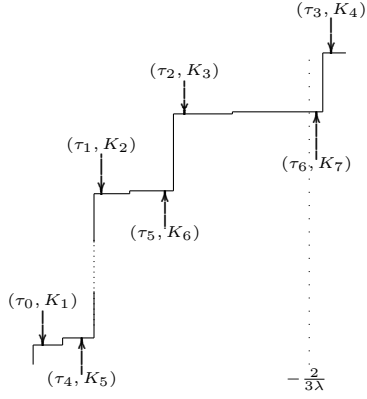

$$
\beginpicture
  \setcoordinatesystem units <0.65 cm, 0.65 cm>

  \putrule from  -9.397 3 to -9.397 3.4
  \putrule from  -9.397 3.4 to -8.794 3.4
  \putrule from  -8.794 3.4 to -8.794 3.545
  \putrule from  -8.794 3.545 to -8.152 3.545
  \putrule from  -8.152 3.545 to -8.152 4.5
  \putrule from  -8.152 5.5 to -8.152 6.497
  \putrule from  -8.152 6.497 to -7.422 6.497
  \putrule from  -7.422 6.497 to -7.422 6.556
  \putrule from  -7.422 6.556 to -6.527 6.556
  \putrule from  -6.527 6.556 to -6.527 8.131
  \putrule from  -6.527 8.131 to  -5.321 8.131
  \putrule from   -5.321 8.131 to -5.321 8.173
  \putrule from   -5.321 8.173 to -3.473 8.173
  \putrule from   -3.473 8.173 to -3.473 9.379
  \putrule from   -3.473 9.379 to -3 9.379

\arrow <4pt> [0.15,0.6] from -9.2 4.1 to -9.2 3.4 \arrow <4pt>
[0.15,0.6] from -8 7.3 to -8 6.497 \arrow <4pt> [0.15,0.6] from
-6.3 8.8 to -6.3 8.131 \arrow <4pt> [0.15,0.6] from -3.3 10.1 to
-3.3 9.379 \arrow <4pt> [0.15,0.6] from -8.4 2.8 to -8.4 3.5459
\arrow <4pt> [0.15,0.6] from -6.7 5.8 to -6.7 6.556 \arrow <4pt>
[0.15,0.6] from -3.6 7.2 to -3.6 8.173

\put{\tiny $(\tau_0,K_1)$} at -9.2 4.3 \put{\tiny $(\tau_1,K_2)$}
at -8 7.5 \put{\tiny $(\tau_2,K_3)$} at -6.3 9 \put{\tiny
$(\tau_3,K_4)$} at -3.3 10.25 \put{\tiny $(\tau_4,K_5)$} at -8.4
2.7 \put{\tiny $(\tau_5,K_6)$} at -6.7 5.7 \put{\tiny
$(\tau_6,K_7)$} at -3.6 7.1 \put{\tiny $-\frac{2}{3\lambda}$} at
-3.8 2.7
\setdots
\putrule from -3.75 3 to -3.75 9.379
\setdots <2pt>
\putrule from  -8.152 3.8 to -8.152 5.6

\endpicture
$$
\caption[$\mathcal{D}$ for $q=9$]{\label{fig3} The sequence
$(\tau_i,K_{i+1})_{i\geq 0}$ in $\Omega_0$ for $k=9$. }
\end{figure}

As in the even case, $\bar{S}$ ``picks up'' a few extra points in
the orbit of $(\tau_0,K_1)$, since both $\tau_{\ell}=-L_1$ and
$\tau_{2\ell}=-L_2$ are larger than $-\frac{2}{3\lambda}$. Similar
to the even case, we find for any $C<C(k)=1/(2\sqrt{\left(
1-\lambda_k/2\right)^2+1})$,
$$
\Theta_n (\tau_0)<C,\quad \text{for at most finitely many $n\geq
0$}.
$$

\section{Equality of Legendre and Lenstra constants}\label{Legendre+Lenstra} 
We define each of the  Legendre  and Lenstra constants for the Rosen mediant maps and show their equality.

\subsection{Definitions}
 Fix an index $k$, and  suppose that there exists $\ell_{k} > 0$ such that
(i)  for any $G_{k}$-irrational $x$ and any finite $G_{k}$-rational $a/c$, 
\[
\left| x \, - \, \frac{a}{c} \right| \, < \, \frac{\ell_{k}}{c^2}
\]
implies $a/c$ is either a Rosen convergent $p_n/q_n$ for some $n \ge
0$, or a mediant Rosen convergent $u_{n,l}/v_{n, l}$ of $x$;  and,
(ii) for any $C > \ell_{k} $, there exist $x$ and $a/c$ such that
\[
\left| x \, - \, \frac{a}{c} \right| \, < \, \frac{C}{c^2}
\]
and $a/c$ is neither a Rosen convergent nor a mediant Rosen
convergent.  Then we call $\ell_{k} > 0$ the {\em Legendre constant} for
mediant Rosen convergents (of index $k$). The Legendre constant
certainly exists for any index $k \ge 4$ because the Legendre
constant for the Rosen continued fractions exists
(see~\cite{R-S}), and the mediant Legendre constant is certainly larger than
or equal to it.\medskip\

We again fix an index $k$, and now suppose that there exists ${\mathcal
L}_{k} > 0$  such that both:   for any $0 < t_1, \, t_2 < {\mathcal
L}_{k}$,  
\begin{equation}
 \lim_{N \to \infty}
 \frac{
  \sharp
  \{1 \le n \le N \, : \, \Theta(M_1 M_2 \cdots M_n (\infty), x) \, < \, t_1 \}            }
{  \sharp
  \{1 \le n \le N \, : \, \Theta(M_1 M_2 \cdots M_n (\infty), x) \, < \, t_2 \}
           }
\, = \, \frac{t_1}{t_2}
\end{equation}
holds for a.e. $x \in {\mathbb J}\,$;  and,  for any $0 < t_2 \, < \,
{\mathcal L}_{k} \, < \, t_1$,
\begin{equation}
 \lim_{N \to \infty}
 \frac{
  \sharp
  \{1 \le n \le N \, : \, \Theta(M_1 M_2 \cdots M_n (\infty), x) \, < \, t_1 \}            }
{  \sharp
  \{1 \le n \le N \, : \, \Theta(M_1 M_2 \cdots M_n (\infty), x) \, < \, t_2 \}
           }
\, < \, \frac{t_1}{t_2}
\end{equation}
for a.e. $x \in {\mathbb J}$.  We call ${\mathcal L}_{k}$ the {\em Lenstra constant}
for the mediant Rosen convergents (of index $k$).    

\subsection{Legendre constant bounded above by Lenstra constant}
 The  following  is a direct
consequence of the corollary of Section~2 in~\cite{Na4}.
\begin{Thm}\label{thm:theorem 4}  Fix any $t_0 > 0$, then for any $t>0$ we have
\begin{equation}\label{HitoshiLimit}
   \lim_{Q \to \infty}\;\; \frac{
   \# \{ \frac{a}{c} \in {G_k}(\infty) \, :
   \, \Theta(\frac{a}{c}, x) < t,
   \quad 0 < c \le Q \} }
  {\# \{\frac{a}{c} \in {G_k}(\infty) \, :
   \, \Theta(\frac{a}{c}, x) < t_0,
   \quad 0 < c \le Q \} }
   \, =  \, \dfrac{t}{t_0}
\end{equation}
for a.e.\ $x$. Here we recall that $G_k (\infty)$ is the set of
parabolic points of the Hecke group $G_{k}\,$.
\end{Thm}
\medskip

From this, we have the following result.
\begin{Prop}  The Legendre constant is less than or equal to the
Lenstra constant, i.e., $\ell_k \le \mathcal L_k$.
\end{Prop}

\begin{proof} Let
\begin{equation}\label{eq:c(n,x,t)}
C(n,x,t)=\# \{ j :\, 1 \le j \le n,\, \Theta (M_1 \cdots M_j
(\infty),x)<t\}.
\end{equation}
If $t$ is smaller than $\ell_k$, then we have
\[
\lim_{N \to \infty} \frac{
   \# \{ \frac{a}{c} \in {G_k}(\infty) \, :
   \, \Theta(\frac{a}{c}, x)< t,
   \quad 0 < c \le q_N \} }
  {\# \{ \frac{a}{c} \in {G_k}(\infty) \, :
   \, \Theta(\frac{a}{c}, x) < t_0,
   \quad 0 < c \le q_N \} }
\, = \, \lim_{n \to \infty} \dfrac{ C(n,x,t)}
  {C(n,x,t_0)}\;,
\]
for almost every $x \in I$.  But this implies that for each such $t$
and  for each of these $x$, the limit as $N$ tends to infinity of
the average of the  counting function, $C(n,x,t)/N$, is a linear
function in $t$. That is, $\mathcal L_k \ge \ell_k$.
\end{proof}
\medskip

\subsection{Harder Inequality}
The idea of the following proof is to begin with an $x_0$ which is
fairly well approximated by some  $G_k$-rational not arising as a
mediant map convergent. We then identify, in terms of the Rosen
fraction expansion of this $x_0$, a whole cylinder set of $x$ all of
which are fairly well approximated by such $G_k$-rationals. There is
then a deficit in the numerator of the fraction of the fundamental
Equation~(\ref{HitoshiLimit}).  This then gives the desired
inequality. We fix $t_0 > 0$ sufficiently small so that it is
smaller than the Legendre constant associated to Rosen continued
fractions.
\begin{Prop}\label{LenSmaller}   Let $C(n,x,t)$ be defined as
in~{\rm (\ref{eq:c(n,x,t)})}. For $t> \ell_k$ and a.e.\ $x\in
{\mathbb I}_k$, one has
\[
\lim_{n \to \infty} \dfrac{C(n,x,t)}{C(n,x,t_0)}< \,
\frac{t}{t_0}\;.
\]
In particular, one has ${\mathcal L}_k \le \ell_k$.
\end{Prop}

\begin{proof} Fix $t$ such that $t>\ell_k$. Then there exist
$x_0 \in I$ and $\frac{p}{q} \in {G_k(\infty)}$ such that
\[
     \Theta\left( x_0,\frac{p}{q}\right) < t\;, \quad \text{with}\;\;
     \frac{p}{q}\;\; \text{unequal to any mediant convergent} \;.
\]
Consider the Rosen fraction expansions:
\[
x_0 = [\,{\ve_1}:{c_1}, \, {\ve_2}:{c_2}, \,  \cdots \,,\,
         {\ve_n}:{c_n},\,  \cdots\;]\;,
\]
and
\[
\frac{p}{q} \, = \, [\, {\ve_1'}:{d_1}, \, {\ve_2'}:{d_2}, \,
       \cdots\,,\, {\ve_l'}:{d_l}]\;.
\]
Since $\frac{p}{q}$ is not a mediant convergent, at least one of
the following does not hold:
\[
\ve_j \, = \, \ve_j ' \; \mbox{for} \; 1 \le j \le l\;;\; c_j \, =
\, d_j \; \mbox{for} \; 1 \le j \le l-1\;; \; \mbox{and}, \; 1 \le
d_l \le c_l - 1\;.
\]

We can choose a large integer $L$ such that
\[
\left| \, y \, - \, \frac{p}{q}\right| <
\frac{t}{q^2}\;\;\;(\,\text{equivalently, such that}\;  \Theta\left(
y,\frac{p}{q}\right) < t\;)
\]
holds whenever
\[
y \, = \,[\, {\ve_1}:{c_1},\,{\ve_2}:{c_2},\, \cdots\,,\,
         {\ve_L}:{c_L},\,{\ve_{L+1}''}:{c_{L+1}''}
           ,\, {\ve_{L+2}''}:{c_{L+2}''} ,\,
          \cdots \,]\;.
\]
We consider all such $y$.  Now suppose that
\[
z \, = \, \confrac{\eta_1}{a_1} + \confrac{\eta_2}{a_2} + \cdots
      + \confrac{\eta_N}{a_N + y}
\]
and
\[
\frac{P}{Q} \, = \,
        \confrac{\eta_1}{a_1} + \confrac{\eta_2}{a_2} + \cdots
      + \confrac{\eta_N}{a_N} + \confrac{p}{q}\; .
\]
For any such pair $(z, \, \frac{P}{Q})$ we have
\[
\left| z \, - \, \frac{P}{Q}\right| < \frac{t}{Q^2}
\]
if $a_N$ is large enough, where the choice of $a_N$ depends only
on $x_0$, $p/q$, and $L$. One checks that $P/Q$ is not a mediant
convergent of $z$ (and thus in particular not a Rosen convergent).
Fix some such $a_N$ and denote it by $a$ and $\eta_N$ by $\eta$.

Let $\mathcal C$ be the cylinder set of all $x$ such that the
initial segment of the Rosen expansion of $x$ matches these:
\[
\mathcal C := \{\, x \in I\ :\, \, x= \, [\,\eta:a,\, \ve_1:c_1
,\, \ve_2:c_2,\, \ldots\,\,\ve_L:c_L,\, \ldots \;]\;\}\;.
\]
By the above discussion, whenever $T^n(x) \in \mathcal C$ there
exists a $G_k$-rational $\frac{P}{Q}$ such that
$$
\left| x -  \frac{P}{Q}\right| < \frac{t}{Q^2} \quad \mbox{and}
\quad q_{n-c} \le Q < q_{n+c}\;,
$$
where $c$ is a constant independent of $n$. Since the Rosen map is
ergodic with respect to the invariant probability measure given
in~\cite{BKS}, the Ergodic Theorem applies and shows that for a.e.
$x \in {\mathbb I}_k$, we have
\[
\lim_{N \to \infty}\;  \dfrac{\# \{ n \le N : \, T^n(x) \in
{\mathcal C}\}}{N}
   \,  = \, \delta\; ,
\]
where $\delta$ is the measure of $\mathcal C$ with respect to the
invariant ergodic measure. In particular, this limit is
positive.\medskip\

Now let $\Xi_N(x)$ be the number of $\frac{P}{Q} \in
{G_k}(\infty)$ such that
\[
\left| x - \frac{P}{Q}\right| < \frac{t}{Q^2}, \quad Q < q_N,
\]
and $P/Q$ is neither convergent nor mediant convergent of $x$. From
the above, we conclude that
\[
\liminf_{N \to \infty} \frac{\Xi_N} { C_0(N,x,t) } > 0
\]
for a.e. $x\in {\mathbb I}_k$, where $C_0(N,x,t) =\#\{ 1\leq n
\leq N :\, \Theta\left( M_1 \cdots M_{k_n} (\infty),x\right) <
t\}$.\smallskip\

Now,
\begin{eqnarray*}
{} & {} & \limsup_{n \to \infty} \frac{C(n, x, t)}{C(n, x, t_0)} \\
{} & \leq & \limsup_{N \to \infty} \frac{ \# \{ \frac{a}{c} \in
{G_k}(\infty) \, : \Theta(x,\frac{a}{c})< t,
   \,\, 0 < c \le q_n\} \,\,-\,\, \Xi_N(x)}
  {\# \{ \frac{a}{c} \in {G_k}(\infty) :
   \, \Theta(x,\frac{a}{c}) < t_0, \, 0 < c \le q_N \} }\\
{} & = & \frac{t}{t_0} \, - \, \liminf_{N\to\infty} \frac{
\Xi_N(x)}{\# \{ \frac{a}{c} \in {G_k}(\infty) :
   \, \Theta(x,\frac{a}{c}) < t_0, \, 0 < c \le q_N \} } \\
{} & \le &
\frac{t}{t_0} \, - \,
\liminf_{N\to\infty} \frac{C_0(N, x, t)}{C_0(N, x, t_0)}\\
{} & \leq & \frac{t}{t_0} - \delta \quad < \quad \frac{t}{t_0}\,.
\end{eqnarray*}
Consequently, the Lenstra constant of the mediant map cannot be
larger than its Legendre constant.
\end{proof}

We have thus demonstrated the equality of the Legendre and Lenstra
constants for the mediant convergents.
\section{Evaluating the Lenstra constant}\label{the-Lenstra-constant}
In this section, we determine the exact value of the Lenstra
constant, and hence of the Legendre also, for the mediant and the
principal Rosen convergents. Note that for the principal Rosen convergents,
the value of the Lenstra constant was stated---without proof---in Corollary~4.1
of~\cite{BKS}.  

\subsection{Reduction to Geometry of Natural Extension} 
First we consider the natural extension $\hat{T}$ of the Rosen
continued fraction map $T$,   defined as follows: the
region of the natural extension $\Omega_0$ is given
by~\eqref{original-Rosen-region-transformed}, and for $(x,y)\in
\Omega_0$, we define
\[
\hat{T} (x, y) \, = \left(
\begin{pmatrix} - r(x) \lambda & \text{sgn}(x) \\ 1 & 0  \end{pmatrix}
 (x),
\begin{pmatrix} - r(x) \lambda & \text{sgn}(x) \\ 1 & 0  \end{pmatrix}
(y) \right) .
\]
This is bijective on $\Omega_0$ a.e., and the absolutely continuous
invariant probability measure is given by
\[
C \cdot \frac{{\rm d}x\, {\rm d}y}{|x - y|^2}
\]
where $C$ is the normalizing constant (see~\cite{BKS} for the
exact value of $C$ in both even and odd cases).

For
$(x,y)\in\Omega_0$, we set $(x_n,y_n)=\hat{T}^n(x, y)$.  In all that follows, we can extend to the setting of $y = -\infty$.   In this case, 
  (\ref{distance}) implies that 
$\theta_{n-1}(x)=1/(x_n-y_n)$.  (The similarity of the denominator of this last with the denominator in the expression for our invariant measure facilitates the following ergodic theoretic approach.)

As we observed in
Section~\ref{Ergodicity}, the measure is ergodic.  By the
individual ergodic theorem and the standard approximation method  (see say Chapter~4 in~\cite{IK}),
we have 
\begin{equation}\label{property(1)}
\lim_{N \to \infty} \frac{1}{N} \sharp \{ n \, : \, 1 \le n \le N,
\, x_n - y_n > t \} \, = \, C \iint_{\{ (x, y) \in \Omega_{0} : x
- y > t \} } \frac{{\rm d}x\, {\rm d}y}{|x - y|^2}
\end{equation}
for any $t > 0$ (a.e. $(x, y) \in \Omega_{0}$).   Elementary calculus  applies to show that the right hand side is equal to $C \lambda \cdot \frac{1}{t}$ if $t$ is
sufficiently large.  By a simple calculation, we see that $|y_n -
y'_n| \to 0$ as $ n \to \infty$ whenever $(x, y), \, (x, y') \in
\Omega_{0}$. This implies that if~(\ref{property(1)}) holds for
$(x, y)$, then it holds for $(x, y')$ too. Thus we get
\eqref{property(1)} for a.e.\ $x \in {\mathbb I}_k$ and any $y$
such that $(x, y) \in \Omega_{0}$; from this, we also have that  the property holds also  for these values of $x$ and with $y=- \infty$. Therefore, we have 
\begin{equation}\label{property(2)}
\lim_{N \to \infty} \frac{1}{N} \sharp\{ n : \, 1 \le n \le N,\,
\theta_n(x) < c \} = C \cdot c
\end{equation}
if $c$ is sufficiently small (a.e.\ $x$). Thus the
Lenstra constant for the Rosen fractions  is the infimum of those
$t>0$, for which
\begin{equation}\label{property(3)}
\iint_{\{ (x, y) \in \Omega_{0} : x - y > t \} } \frac{{\rm d}x\,
{\rm d}y}{|x - y|^2} = \frac{\lambda}{t}
\end{equation}
holds. It is easily seen, compare with Figures ~\ref{figOm8} and ~\ref{sHatIm9},  that this can be determined by the infimum
$t_0$ of those $t > 0$, for which the points on the line segment
\begin{equation}\label{property(4)}
y = x + t, \quad x \in {\mathbb I}_k,
\end{equation}
are all in $\Omega_{0}$.

Now for the mediant Rosen convergents,  analogous arguments apply.
We use the ratio ergodic theorem, see say Ch.~3 of \cite{A}, instead of the individual
ergodic theorem, and we obtain the completely analogous conclusion, i.e.,  $\mathcal L_k$
is the infimum $t_1$ of $t$ such that
\begin{equation}\label{property(5)}
\iint_{\{ (x, y) \in \Omega^{\ast} : x - y > t \} } \frac{{\rm
d}x\, {\rm d}y}{|x - y|^2} \, = \, \frac{\lambda}{t}
\end{equation}
holds.\medskip\
\noindent We consider the even and the odd indices cases
separately.

\subsection{Even index case: $k = 2\ell$} First we show that the
Lenstra constant for the Rosen convergent is
$\frac{\lambda}{\lambda + 2}$, confirming Corollary~4.1
of~\cite{BKS}.\smallskip\

To find $t_0$ with property~(\ref{property(4)}), it is enough to
check the lines of slope $1$ passing through the interior corners of $\Omega_{0}$.  The associated equations
are
\[
\left\{
\begin{array}{ll}
y = x - (\frac{1}{L_j} + \phi_j) &  1 \le j \le \ell-1 \\
y = x - (1 - \frac{\lambda}{2}). & {}
\end{array} \right.
\]
From these, we see that $t_0 \, = \, \max \{ \frac{1}{L_j} +
\phi_j \, (1 \le j \le \ell-1), \, 1 - \frac{\lambda}{2} \}$.
Since
\[
L_j = \frac{1}{\lambda - L_{j-1}} \qquad \text{and} \qquad
\phi_{j} = - \frac{1}{\phi_{j-1}} - \lambda ,
\]
it follows that
\[
\frac{1}{L_j} + \phi_j = \dfrac{\frac{1}{L_{j-1}} + \phi_{j-1}}
                                {\left|   \phi_{j-1}/L_{j-1} \right|}
\]
for $2 \le j \le \ell-1$. Because $\left(
1/L_{j}\right)_{j=0}^{\ell-1}$ and $\left( |\phi_{j}|
\right)_{j=0}^{\ell-1}$ are monotonically decreasing sequences,
the maximum is either $\frac{1}{L_1} + \phi_1$,
$\frac{1}{L_{\ell-1}} + \phi_{\ell-1}$, or $1-\lambda/2$. Now
recall that $\phi_{\ell-1} = 0$ and $L_{\ell-1} = \lambda -1$ and
$\lambda \ge \sqrt{2}$. These yield the estimate $\frac{\lambda +
2}{\lambda} = \frac{1}{L_1} + \phi_1 \, \ge \,
\frac{1}{L_{\ell-1}} + \phi_{\ell-1}$. Note that it is easy to
show that $\frac{\lambda + 2}{\lambda} = \frac{1}{L_1} + \phi_1
> 1 -\frac{\lambda}{2}$. Consequently, we have that $t_0 =
\frac{1}{L_1} + \phi_1 = \frac{\lambda + 2}{\lambda}$.  And, the result holds. \smallskip\

The result for the mediant Rosen
convergents is the following.

\begin{Prop}\label{prop:LenstaConstantEvenCase}
The Lenstra constant for the mediant Rosen convergent is $\lambda
-1$ when the index is even and not equal to $4$. If $k=4$, then
the Lenstra constant is equal to $\sqrt{2}/2$.
\end{Prop} 
\begin{proof} It is obvious that the measure $\frac{{\rm d}x\, {\rm
d}y}{|x - y|^2}$ is invariant under the translation $(x, y)
\mapsto (x + z, y + z)$ for any real number $z$. We translate the
set $J_{l+1} \times \bar{K}_{l+1}$ by $-\lambda$. Then the image
is $[\phi_{0}, \phi_{1}) \times [-\frac{1}{L_1}, -\frac{1}{L_1} +
1) \, = \, [\phi_{0}, \phi_{1}) \times [-\lambda - 1, \,
-\lambda)$ and we see that $-\lambda < -\frac{1}{L_2} = -\lambda +
\frac{1}{\lambda + R}$. This shows that for the mediant case, we
can get $t_1$ by $\max \{ -\lambda + \phi_1 , \frac{1}{L_j} +
\phi_j , \, (2 \le j \le \ell-1), \, \frac{\lambda}{2} \}$.  Similarly to the above, the maximum is given by either $\lambda + \phi_1,
\, \frac{1}{L_2} + \phi_2$, $\frac{1}{L_{\ell-1}} + \phi_{\ell-1}
= \frac{1}{\lambda - 1}$, or $\frac{\lambda}{2}$. Thus we get $t_1
= \frac{1}{\lambda - 1}$ when $\ell \ge 3$; see Figure~\ref{fig4}.
If $l = 2$, a simple calculation shows that $t_1 = \sqrt{2} + 1$.
\end{proof}

\begin{figure}[h]
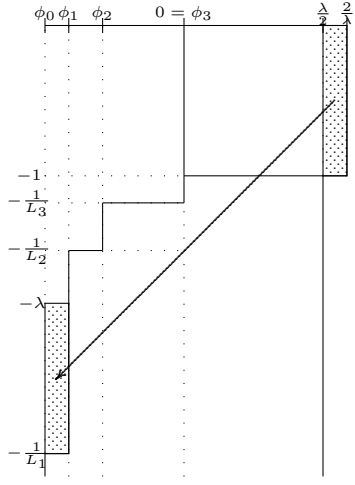

$$
\beginpicture
    \setcoordinatesystem units <0.2cm,0.2cm>
    \arrow<4pt> [.2,.8] from 10 -5 to -8.5 -23.54
    \setplotarea x from -12 to 12, y from -29 to 1
    \put {\tiny{$0=\phi_3$}} at 0 0.8
    \put {\tiny{$\phi_0$}} at -9.24 0.8
    \put {\tiny{$\phi_1$}} at -7.66 0.8
    \put {\tiny{$\phi_2$}} at -5.41 0.8
    \put {\tiny{$\frac{\lambda}{2}$}} at 9.24 0.8
    \put {\tiny{$\frac{2}{\lambda}$}} at 10.83 0.8
    \put {\tiny{$-\frac{1}{L_1}$}} at -10.3  -28.6
    \put {\tiny{$-\frac{1}{L_2}$}} at -10.3  -15
    \put {\tiny{$-\frac{1}{L_3}$}} at -10.3  -11.8
    \put {\tiny{$-1$}} at -10.3  -10
    \put {\tiny{$-\lambda$}} at -10.1 -18.48
    \putrule from -9.24 0 to 10.83 0
    \putrule from -9.24 -30 to -9.24 -28.48
    \putrule from -9.24 -0.5 to -9.24 0.5
    \putrule from -9.24 -28.48 to -7.66 -28.48
    \putrule from -7.66 -28.48 to -7.66 -14.97
    \putrule from -7.66 -0.5 to -7.66 0.5
    \putrule from -7.66 -14.97 to -5.41 -14.97
    \putrule from -5.41 -14.97 to -5.41 -11.8
    \putrule from -5.41 -0.5 to -5.41 0.5
    \putrule from -5.41 -11.8 to 0 -11.8
    \putrule from 0 -11.8 to 0 0.5
    \putrule from 0 -10 to 10.83 -10
    \putrule from 10.83 -10 to 10.83 0.5
    \putrule from 9.24 -30 to 9.24 -10
    \putrule from 9.24 -0.5 to 9.24 0.5
\setshadesymbol <.1pt,.1pt,.1pt,.1pt> ({\fiverm .})
  \shaderectangleson
  \setshadegrid span <1.5pt>
  \putrectangle corners at -9.24 -28.48 and -7.66 -18.48
  \putrectangle corners at 9.24 -10 and 10.83 0
\shaderectanglesoff

\setdots \putrule from -9.24 -28.48 to -9.24 0 \putrule from -7.66
-30 to -7.66 0 \putrule from -5.41 -30 to -5.41 0 \putrule from 0
-30 to 0 -11.7 \putrule from 9.24 -10 to 9.24 0 \putrule from
-9.24 -10 to 0 -10 \putrule from -9.24 -11.8 to 0 -11.8 \putrule
from -9.24 -14.97 to 0 -14.97
\endpicture
$$ \caption[nat-ext]{The translation of the set $J_{l+1}\times \bar{K}_{l+1}$ by $-\lambda$. Here $k=8$.}
\label{fig4}
\end{figure}

\subsection{Odd index case: $k = 2\ell + 3$}
Here also we first confirm Corollary~4.1 of~\cite{BKS}:  the Lenstra constant of the Rosen convergents  equals   $\frac{R}{R + 1}$.

The idea of the calculation is the same as the even case.   Considering the slope $1$ lines through the corners of $\Omega_{0}$,   we find that
\[
t_0 \, = \, \max \left\{ \frac{1}{L_{2j}} + \phi_j \, (1 \le j \le
\ell), \, \frac{1}{L_{2j-1}} + \phi_{\ell+j} , \, (1 \le j \le
\ell+1), \, \frac{1}{R} + \frac{\lambda}{2} \right\}.
\]
Since
\[
\phi_{j+1} = -\frac{1}{\phi_j} - \lambda \quad \text{for} \quad 0
\le j < \ell, \, \ell+1 \le j < 2\ell + 1
\]
and
\[
L_{j+2} = \frac{1}{\lambda_j - L_j} \quad \text{for} \quad 1 \le j
\le 2\ell - 1,
\]
we have
\[
\phi_{j+1} + \frac{1}{L_{2(j+1)}} \, = \, \frac{\phi_{j} +
\frac{1}{L_{2j}}}{-\phi_j \cdot \frac{1}{L_{2j}}} \qquad
\text{and} \qquad \phi_{\ell+j+1} + \frac{1}{L_{2j+)}} \, = \,
\frac{\phi_{\ell+j} + \frac{1}{L_{2j-1}}}{-\phi_{l+j} \cdot
\frac{1}{L_{2j-1}}}
\]
for $1 \le j \le \ell-1$. Moreover, $\phi_{\ell+1} ( = 1 -
\lambda) = -\frac{1}{\phi_l} - 2 \lambda$ and $L_1 =
\frac{1}{2\lambda - L_{\ell}}$ implies
\[
 \phi_{\ell} + \frac{1}{L_{2\ell}} \, = \,
\frac{\phi_{\ell} + \frac{1}{L_{2\ell}}}{-\phi_{\ell} \cdot
\frac{1}{L_{2}}}\; .
\]
Again, $\left(|\phi_j| \, : \, 1 \le j \le \ell+1 \right)$,
$\left(|\phi_{\ell+j}| \, : \, 1 \le j \le \ell+1 \right)$, and
$\left( \frac{1}{L_j} \, : \, 1 \le j \le 2\ell-1 \right)$ are
decreasing sequences. So the above maximum is equal to
\begin{eqnarray*}
{} & {}& \max\left\{ \phi_1 \, + \, \frac{1}{L_2}, \,
\phi_{\ell+1} \, + \, \frac{1}{L_1}, \, \phi_{2\ell+1} \, + \,
\frac{1}{L_{2\ell+1}}, \, \frac{1}{R} \, + \, \frac{\lambda}{2}
\right\} \\
{} & = & \max\left\{ \frac{2}{\lambda} + R , \, \frac{R+1}{R}, \,
\frac{1}{\lambda - R}, \, \frac{1}{R} +  \frac{\lambda}{2}
\right\} .
\end{eqnarray*}
Due to the facts that $R^2 + (2 - \lambda)R -1 = 0$ and $\lambda
/2 < R < 1$, we see that the maximum is equal to $\frac{R+1}{R}$,
and the result follows.\smallskip\

For the mediant Rosen convergents, 
we have the following.
\begin{Prop}\label{prop:LenstaConstantOddCase}
In case of odd index $k$, the Lenstra constant for the mediant
Rosen convergents is $\lambda - R$.
\end{Prop}

\begin{proof}
We translate
$(J_{2\ell+3} \times \bar{K}_{2\ell+3}) \cup (J_{2\ell+4} \times
\bar{K}_{2\ell+4})$ by $-\lambda$; see Figure~\ref{fig5}.  Then
its image is
\[
[- \frac{\lambda}{2} , \, \phi_{\ell+1}) \times [-\frac{1}{R} -
\lambda, \, -\lambda)
  \cup
[\phi_{\ell+1}, \, \frac{2}{\lambda} - \lambda) \times [ -R -
\lambda, \, -\lambda)
\]
\begin{figure}[h]
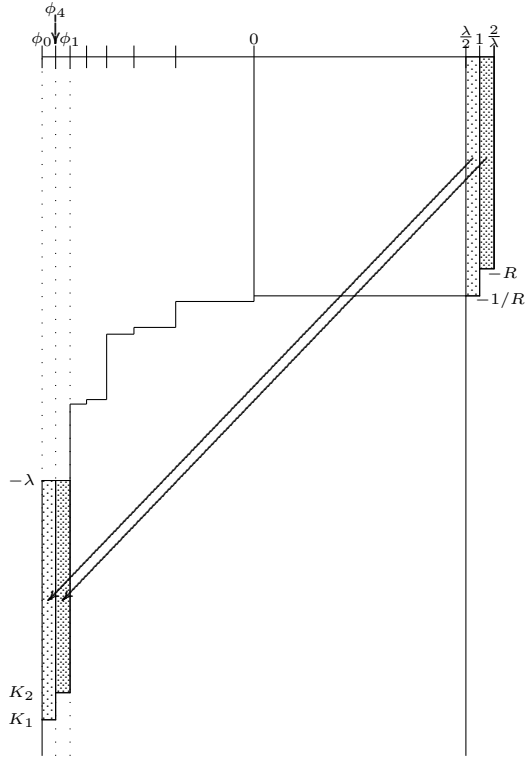

$$
\beginpicture
    \setcoordinatesystem units <0.3cm,0.3cm>
    \setplotarea x from -15 to 15, y from -31 to 2
    \put {\tiny{$0$}} at 0 0.8
    \put {\tiny{$\phi_0$}} at -9.39 0.8
    \put {\tiny{$\phi_1$}} at -8.15 0.8
    \put {\tiny{$\phi_4$}} at -8.79 2.1
    \arrow<4pt> [.2,.8] from -8.79 1.85 to -8.79 0.6
    \put {\tiny{$\frac{\lambda}{2}$}} at 9.39 0.9
    \put {\tiny{$-\lambda$}} at -10.3 -18.79
    \put {\tiny{$\frac{2}{\lambda}$}} at 10.64 0.9
    \put {\tiny{$K_1$}} at -10.3  -29.4
    \put {\tiny{$K_2$}} at -10.3  -28.2
    \put {\tiny{$-1/R$}} at 10.9  -10.8
    \put {\tiny{$-R$}} at 11  -9.6
    \put {\tiny{$1$}} at 10  0.8
    \putrule from 10 -0.5 to 10 0.5
    \putrule from -9.39 0 to 10.64 0
    \putrule from -9.39 -31 to -9.39 -29.4
    \putrule from -9.39 -0.5 to -9.39 0.5
    \putrule from -9.39 -29.4 to -8.79 -29.4
    \putrule from -8.79 -29.4 to -8.79 -28.2
    \putrule from -8.79 -0.5 to -8.79 0.5
    \putrule from -8.79 -28.2 to -8.15 -28.2
    \putrule from -8.15 -28.2 to -8.15 -15.4
    \putrule from -8.15 -0.5 to -8.15 0.5
    \putrule from -8.15 -15.4 to -7.42 -15.4
    \putrule from -7.42 -0.5 to -7.42 0.5
    \putrule from -7.42 -15.4 to -7.42 -15.2
    \putrule from -7.42 -15.2 to -6.53 -15.2
    \putrule from -6.53 -15.2 to -6.53 -12.3
    \putrule from -6.53 -0.5 to -6.53 0.5
    \putrule from -6.53 -12.3 to -5.32 -12.3
    \putrule from -5.32 -12.3 to -5.32 -12
    \putrule from -5.32 -0.5 to -5.32 0.5
    \putrule from -5.32 -12 to -3.47 -12
    \putrule from -3.47 -0.5 to -3.47 0.5
    \putrule from -3.47 -12 to -3.47 -10.85
    \putrule from -3.47 -10.85 to 0 -10.85
    \putrule from 0 -10.85 to 0 0.5
    \putrule from 0 -10.6 to 10 -10.6
    \putrule from 10.64 -9.4 to 10.64 0.5
    \putrule from 10 -10.6 to 10 -9.4
    \putrule from 10 -9.4 to 10.64 -9.4
    \putrule from 10.64 -9.4 to 10.64 0.5
    \putrule from 9.39 -31 to 9.39 -10.6
    \putrule from 9.39 -0.5 to 9.39 0.5
\setshadesymbol <.1pt,.1pt,.1pt,.1pt> ({\fiverm .})
  \shaderectangleson
  \setshadegrid span <1.5pt>
  \putrectangle corners at -9.39 -29.4 and -8.79 -18.79
  \putrectangle corners at 9.39 -10.6 and 10 0
\shaderectanglesoff \setshadesymbol <.1pt,.1pt,.1pt,.1pt> ({\fiverm
.})
  \shaderectangleson
  \setshadegrid span <1pt>
  \putrectangle corners at -8.79 -28.2 and -8.15 -18.79
  \putrectangle corners at 10 -9.4 and 10.64 0
\shaderectanglesoff \arrow<4pt> [.2,.8] from 9.7 -4.5 to -9.1 -24.1
\arrow<4pt> [.2,.8] from 10.3 -4.5 to -8.47 -24.1

\setdots \putrule from -9.39 -29.4 to -9.39 0 \putrule from -8.79
-31 to -8.79 0 \putrule from -8.15 -31 to -8.15 0
\endpicture
$$ \caption[nat-ext]{The translation of the set $(J_{2l+3} \times \bar{K}_{2l+3}) \cup (J_{2l+4} \times
\bar{K}_{2l+4})$ by $-\lambda$. Here $k=9$.} \label{fig5}
\end{figure}
This just fits on
$$
J_1 \times \hat{K}_1 \cup J_2 \times \hat{K}_2 =[-
\frac{\lambda}{2} , \, \phi_{\ell+1}) \times [-\infty , \,
-\frac{1}{R} - \lambda)
  \cup
[\phi_{\ell+1}, \, \frac{2}{\lambda} - \lambda) \times [ - \infty,
\, -R - \lambda) ,
$$
(note that $\phi_{\ell+1} = 1 - \lambda$).  Now we find that the
maximum $t_0$ in the above is cancelled by this justification and
have the new value $\phi_1 + \lambda = \frac{2}{\lambda}$ because
$-\lambda < - \frac{1}{L_3} = -\lambda + \frac{1}{\lambda +
\frac{1}{R}}$. Then we have
\[
t_1 \, = \, \max\left\{ \frac{2}{\lambda}, \, \frac{1}{L_{2j}} +
\phi_j \, (2 \le j \le \ell), \, \frac{1}{L_{2j-1}} +
\phi_{\ell+j} \, (2 \le j \le \ell+1), \, \frac{\lambda}{2}
\right\} .
\]
This is the same as
\[
\max\left\{ \frac{2}{\lambda}, \, \frac{1}{L_{4}} + \phi_2 , \,
\frac{1}{L_{2\ell}} + \phi_{\ell} , \, \frac{1}{L_{3}} +
\phi_{\ell+2} , \, \frac{1}{L_{2\ell+1}} + \phi_{2\ell+1} , \,
\frac{\lambda}{2} \right\} .
\]
One has the following relations,
$$
\begin{array}{lcl}
\frac{1}{L_{4}} + \phi_2 & = &
\frac{1}{\lambda - \frac{2}{\lambda}} - \frac{1}{\lambda + R} \\
{}&{}&{} \\
\frac{1}{L_{2\ell}} + \phi_{\ell} & = &
\frac{R}{\lambda R  - 1} - \frac{1}{\lambda + 1} \\
{}&{}&{} \\
\frac{1}{L_{3}} + \phi_{\ell+2} & = &
\frac{1}{\lambda - 1} - \frac{R}{\lambda R + 1} \\
{}&{}&{} \\
\frac{1}{L_{2\ell+1}} + \phi_{2\ell+1} & = & \frac{1}{\lambda - R\;.
}
\end{array}
$$
Here $\frac{1}{L_{4}} + \phi_2$ and $\frac{1}{L_{2\ell}} +
\phi_{\ell}$ do not appear in the above when $k=5$. After some
calculation, we see the maximum is $\frac{1}{\lambda - R}$. In
order to see that this is indeed the case, note that we obviously
have that $\lambda /2 < 2/\lambda$, and that
\[
\frac{2}{\lambda} < \frac{1}{\lambda - R}
\] follows from $\lambda /2 < R $. Since
$R^2 + (2 - \lambda)R - 1 = 0$ and $\lambda \ge \frac{1 +
\sqrt{5}}{2}\,$, we have
\[
\frac{1}{\lambda - 1} - \frac{R}{\lambda R + 1} \le
\frac{1}{\lambda - R\,} .
\]

We see
\[
\frac{1}{\lambda - \frac{2}{\lambda}} - \frac{1}{\lambda + R} <
\frac{1}{\lambda - R\,} ,
\]
when $k > 5$. Here we used the fact that $\lambda^2 > 4 - R^2$ for
$k > 5$, which has to be checked somehow. Because $\lambda$ and
$R$ is increasing as $k$ increases, it is sufficient to prove it
for $k=7$.\medskip\

Finally we show that
\[
\frac{R}{\lambda R  - 1} \, - \, \frac{1}{\lambda + 1} \, < \,
\frac{1}{\lambda - \frac{2}{\lambda}} \, - \, \frac{1}{\lambda +
R\,} .
\]
for $k \ge 7$.  This inequality is equivalent to
\[
\frac{1}{\lambda - \frac{1}{R}} \, + \, \frac{1}{\lambda + R} \, <
\, \frac{\lambda}{\lambda ^2 - 2} \, + \, \frac{1}{\lambda + 1}\,.
\]
Since $R-1/R = \lambda - 2$, this is equivalent to
\[
\frac{3\lambda - 1}{\lambda^2 + (\lambda - 2) \lambda - 1} \, < \,
\frac{2\lambda^2 + \lambda - 2}{\lambda^3 + \lambda^2 - 2 \lambda
- 2\,} .
\]
Note that both denominators are positive (for $\lambda >
\sqrt{3}$). The last inequality follows from
\[
\lambda^4 - 3 \lambda^3 + 5\lambda - 2 > 0 \qquad \text{for
$\lambda > \sqrt{3}\,$}.
\]
\end{proof} 



\begin{thebibliography}{99}
\bibitem{A} J.~ Aaronson,  An Introduction to Infinite Ergodic Theory,  AMS, 1997.  

\bibitem{BJW}W.~Bosma, H.~Jager, and F.~Wiedijk, \emph{Some metrical
observations on the approximation by continued fractions}, Indag.\
Math.\ 45 (1983), 353--379.
 
\bibitem{BKS}R.~Burton, C.~Kraaikamp, and T.A.~Schmidt, \emph{Natural
extensions for the Rosen fractions}, TAMS 352 (1999), 1277--1298.

\bibitem{DK}K.~Dajani, and C.~Kraaikamp, Ergodic theory of numbers,
Carus Mathematical Monographs, 29. Mathematical Association of
America, Washington, DC, 2002.
 
\bibitem{H}A.~Haas, \emph{The distribution of geodesic excursions out the end of a hyperbolic orbifold and approximation with respect to a Fuchsian group}, 
Geom. Dedicata 116 (2005), 129--155.

\bibitem{Ha-Ser}A.~Haas and C.~Series, \emph{The Hurwitz constant and
Diophantine approximation on Hecke groups}, J.\ London Math.\ Soc.
{\bf 34} (1986), 219--234.

\bibitem{I} S.~Ito, \emph{Algorithms with mediant convergents
and their metrical theory}, Osaka J.~Math.\ 26 (1989), no.\ 3,
557--578.

\bibitem{IK}M.~Iosifescu, and C.~Kraaikamp, Metrical Theory of Continued Fractions,
Mathematics and its Applications, 547. Kluwer Academic Publishers,
Dordrecht, 2002.
--39.

\bibitem{krengel}U.~Krengel, \emph{Entropy of conservative transformations},
Z.\ Wahrscheinlichkeitstheorie und Verw.\ Gebiete 7 (1967), 161--181.

\bibitem{KSS}C.~Kraaikamp, T.A.~Schmidt, and I.~Smeets \emph{Tong's
Spectrum for Rosen Continued Fractions},  J.\ de
Th\'eor. des Nombres Bordeaux 19  (2007),  no. 3, 641--661. 

\bibitem{LeBk}J.~Lehner, \emph{Discontinuous Groups and Automorphic Functions}, AMS Mathematical
Surveys and Monographs, vol 8 , 1964.


\bibitem{Le1} \bysame, \emph{Diophantine approximation on Hecke
groups}, Glasgow Math.\ J.\ 27 (1985), 117--127.

\bibitem{Le2} \bysame, \emph{The local Hurwitz constant and diophantine
approximation on Hecke groups}, Math.\ Comp.\ 55 (1990), 765--781.

\bibitem{NIT} H.~Nakada, S.~Ito, and S.~Tanaka, \emph{On the invariant
measure for the transformations associated with
some real continued-fractions}, Keio Engrg.\ Rep.\ 30 (1977), no.\
13, 159--175.
 
\bibitem{Na1} H.\ Nakada, \emph{Metrical theory for a class of
continued fraction transformations}, Tokyo J.\ Math.\ {\bf 4} (1981),
399--426.

\bibitem{Na3} \bysame, \emph{Continued fractions, geodesic flows and Ford circles},
in  Algorithms, Fractals and Dynamics edited by Y.~Takahashi,
179--191, Plenum, 1995.

\bibitem{Na4} \bysame, \emph{On the Lenstra constant associated to the Rosen
continued fractions}, to appear in J. Eur.
Math. Soc. 

\bibitem{Nt}R.~Natsui, \emph{On the interval maps associated to the $\alpha$-mediant
convergents}, Tokyo J.\ Math.\ 27 (2004), 87--106.

\bibitem{Roh} V.A.\ Rohlin, \emph{Exact endomorphisms of
Lebesgue spaces}, Izv.\ Akad.\ Nauk SSSR Ser.\ Mat.\ \textbf{25}
(1961), 499--530. Amer.\ Math.\ Soc.\ Transl.\ Series 2, {\bf 39}
(1964), 1--36.


\bibitem{R}D.~Rosen, \emph{A class of continued fractions associated with certain
properly discontinuous groups}, Duke Math.\ J.\ {\bf21} (1954),
549--563.

\bibitem{R-S}D.~Rosen, and T.A.~Schmidt, \emph{Hecke groups and continued fractions},
Bull.\ Austral.\ Math.\ Soc.\ 46 (1992), 459--474.

\bibitem{Schw} F.\ Schweiger, \emph{Ergodic theory of fibred
systems and metric number theory}. Oxford: Clarendon Press, 1995.

 
\end{thebibliography}
\end{document}